\documentclass[12pt,fleqn]{article}

\usepackage{amsmath,amssymb,amsthm}
\usepackage{geometry} 
\usepackage[pdfpagemode=UseNone,pdfstartview=FitH]{hyperref}

\newcommand{\Ap}[1][]{A_p\, #1}
\newcommand{\abs}[1]{\left|#1\right|}
\newcommand{\Bp}[1][]{B_p\, #1}
\newcommand{\bdry}[1]{\partial #1}

\newcommand{\F}{{\cal F}}
\newcommand{\closure}[1]{\overline{#1}}
\newcommand{\comp}{\circ}
\newcommand{\dint}{\ds{\int}}
\newcommand{\dist}[2]{\text{dist}\, (#1,#2)}
\newcommand{\dnorm}[2][]{\left\|#2\right\|_{#1}^\ast}
\newcommand{\ds}[1]{\displaystyle #1}
\newcommand{\dualp}[3][]{\left(#2,#3\right)_{#1}}
\newcommand{\eps}{\varepsilon}
\newcommand{\half}{\frac{1}{2}}
\newcommand{\hquad}{\hspace{0.08in}}

\newcommand{\loc}{\text{loc}}
\newcommand{\M}{{\cal M}}
\newcommand{\minuszero}{\! \setminus\! \set{0}}
\newcommand{\N}{\mathbb N}
\newcommand{\norm}[2][]{\left\|#2\right\|_{#1}}
\renewcommand{\o}{\text{o}}
\newcommand{\PS}[1]{$(\text{PS})_{#1}$}
\newcommand{\pnorm}[2][]{\if #1'' \left|#2\right|_p \else \left|#2\right|_{#1} \fi}
\newcommand{\QED}{\mbox{\qedhere}}
\newcommand{\R}{\mathbb R}

\newcommand{\restr}[2]{\left.#1\right|_{#2}}
\newcommand{\seq}[1]{\left(#1\right)}
\newcommand{\set}[1]{\left\{#1\right\}}
\newcommand{\vol}[1]{\left|#1\right|}

\newcommand{\Z}{\mathbb Z}

\DeclareMathOperator{\divg}{div}

\DeclareMathOperator{\supp}{supp}

\newenvironment{enumroman}{\begin{enumerate}
\renewcommand{\theenumi}{$(\roman{enumi})$}
\renewcommand{\labelenumi}{$(\roman{enumi})$}}{\end{enumerate}}

\newtheorem{corollary}{Corollary}[section]
\newtheorem{lemma}[corollary]{Lemma}

\newtheorem{theorem}[corollary]{Theorem}

\theoremstyle{definition}

\theoremstyle{remark}

\newtheorem{remark}[corollary]{Remark}

\numberwithin{equation}{section}

\title{\bf A general perturbation theorem with applications to nonhomogeneous critical growth elliptic problems}
\author{\bf Kanishka Perera\\
Department of Mathematical Sciences\\
Florida Institute of Technology\\
Melbourne, FL 32901, USA\\
\em kperera@fit.edu}
\date{}

\begin{document}

\maketitle

\begin{abstract}
We prove a general perturbation theorem that can be used to obtain pairs of nontrivial solutions of a wide range of local and nonlocal nonhomogeneous elliptic problems. Applications to critical $p$-Laplacian problems, $p$-Laplacian problems with critical Hardy-Sobolev exponents, critical fractional $p$-Laplacian problems, and critical $(p,q)$-Laplacian problems are given. Our results are new even in the semilinear case $p = 2$.
\end{abstract}

\newpage

\section{Introduction}

In the pioneering paper \cite{MR1168304}, Tarantello showed that the problem
\begin{equation} \label{0.1}
\left\{\begin{aligned}
- \Delta u & = |u|^{2^\ast - 2}\, u + h(x) && \text{in } \Omega\\[10pt]
u & = 0 && \text{on } \bdry{\Omega},
\end{aligned}\right.
\end{equation}
where $\Omega$ is a bounded domain in $\R^N$, $N \ge 3$, and $2^\ast = 2N/(N - 2)$ is the critical Sobolev exponent, has two nontrivial solutions if $h \in H^{-1} \minuszero$ satisfies
\[
\int_\Omega hu\, dx < \frac{4}{N - 2} \left(\frac{N - 2}{N + 2}\right)^{(N+2)/4} \pnorm[2]{\nabla u}^{(N+2)/2}
\]
for all $u \in H^1_0(\Omega)$ with $\pnorm[2^\ast]{u} = 1$, where $\pnorm[p]{\cdot}$ denotes the norm in $L^p(\Omega)$. In particular, problem \eqref{0.1} has two nontrivial solutions for all $h \in L^{2N/(N+2)}(\Omega) \minuszero$ with $\pnorm[2N/(N+2)]{h}$ sufficiently small. In \cite{MR1408672}, Cao and Zhou extended this result to the problem
\begin{equation} \label{0.2}
\left\{\begin{aligned}
- \Delta u & = \lambda u + |u|^{2^\ast - 2}\, u + h(x) && \text{in } \Omega\\[10pt]
u & = 0 && \text{on } \bdry{\Omega}
\end{aligned}\right.
\end{equation}
for $0 < \lambda < \lambda_1$, where $\lambda_1 > 0$ is the first Dirichlet eigenvalue of $- \Delta$ in $\Omega$. There is now a large literature generalizing these results (see, e.g., \cite{MR4122508,MR1966255,MR1975059,MR3250758,MR2231068,MR2488689,MR3649469} and their references). However, the question of whether problem \eqref{0.2} still has two nontrivial solutions for all $h \in L^{2N/(N+2)}(\Omega) \minuszero$ with $\pnorm[2N/(N+2)]{h}$ sufficiently small when $\lambda \ge \lambda_1$ has remained open over the years. In the present paper we show that this is indeed the case when $N = 4$ and $\lambda > \lambda_1$ is not an eigenvalue, and when $N \ge 5$ and $\lambda \ge \lambda_1$. More specifically, we have the following theorem.

\begin{theorem} \label{Theorem 0.1}
There exists $\mu_0 > 0$ such that problem \eqref{0.2} has two nontrivial solutions for all $h \in L^{2N/(N+2)}(\Omega) \minuszero$ with $\pnorm[2N/(N+2)]{h} < \mu_0$ in each of the following cases:
\begin{enumroman}
\item $N = 4$ and $\lambda > 0$ is not an eigenvalue,
\item $N \ge 5$ and $\lambda > 0$.
\end{enumroman}
\end{theorem}

We will in fact prove the corresponding result for the $p$-Laplacian. Consider the problem
\begin{equation} \label{1.1}
\left\{\begin{aligned}
- \Delta_p\, u & = \lambda\, |u|^{p-2}\, u + |u|^{p^\ast - 2}\, u + h(x) && \text{in } \Omega\\[10pt]
u & = 0 && \text{on } \bdry{\Omega},
\end{aligned}\right.
\end{equation}
where $\Omega$ is a bounded domain in $\R^N$, $N \ge 2$, $\Delta_p\, u = \divg (|\nabla u|^{p-2}\, \nabla u)$ is the $p$-Laplacian of $u$, $1 < p < N$, $p^\ast = Np/(N - p)$ is the critical Sobolev exponent, $\lambda > 0$, $h \in L^{{p^\ast}'}(\Omega) \minuszero$, and ${p^\ast}' = p^\ast/(p^\ast - 1)$ is the H\"{o}lder conjugate of $p^\ast$. We have the following theorem.

\begin{theorem} \label{Theorem 1.1}
There exists $\mu_0 > 0$ such that problem \eqref{1.1} has two nontrivial solutions for all $h \in L^{{p^\ast}'}(\Omega) \minuszero$ with $\pnorm[{p^\ast}']{h} < \mu_0$ in each of the following cases:
\begin{enumroman}
\item $N \ge p^2$ and $\lambda > 0$ is not a Dirichlet eigenvalue of $- \Delta_p$ in $\Omega$,
\item $N^2/(N + 1) > p^2$ and $\lambda > 0$.
\end{enumroman}
\end{theorem}

\begin{remark}
When $h = 0$, one nontrivial solution of problem \eqref{1.1} was obtained by Garc{\'{\i}}a Azorero and Peral Alonso \cite{MR912211}, Egnell \cite{MR956567}, Guedda and V{\'e}ron \cite{MR1009077}, Arioli and Gazzola \cite{MR1741848}, and Degiovanni and Lancelotti \cite{MR2514055}.
\end{remark}

We will first prove a general perturbation theorem that can be used to obtain pairs of nontrivial solutions of a wide range of local and nonlocal nonhomogeneous elliptic problems (see Theorem \ref{Theorem 1.5}). We will then apply this result to prove Theorem \ref{Theorem 1.1} for a more general class of critical $p$-Laplacian problems (see Theorem \ref{Theorem 1.6}). We also present applications of our general result to $p$-Laplacian problems with critical Hardy-Sobolev exponents (Theorem \ref{Theorem 1.7}), critical fractional $p$-Laplacian problems (Theorem \ref{Theorem 1.9}), and critical $(p,q)$-Laplacian problems (Theorem \ref{Theorem 1.10}).

Proof of Theorem \ref{Theorem 1.5} makes use of a certain linking structure associated with a sequence of eigenvalues introduced by the author in \cite{MR1998432}. This sequence is defined using the $\Z_2$-cohomological index of Fadell and Rabinowitz (see \cite{MR0478189}). We point out that the standard sequence of eigenvalues defined using the genus does not provide this linking structure and therefore cannot be used to prove Theorem \ref{Theorem 1.5}.

\section{Statement of results}

\subsection{A general perturbation theorem}

Let $(W,\norm{\, \cdot\,})$ be a uniformly convex Banach space with dual $(W^\ast,\dnorm{\, \cdot\,})$ and duality pairing $\dualp{\cdot}{\cdot}$. Recall that $f \in C(W,W^\ast)$ is a potential operator if there is a functional $F \in C^1(W,\R)$, called a potential for $f$, such that $F' = f$. We consider the nonlinear operator equation
\begin{equation} \label{1.3}
\Ap[u] = \lambda \Bp[u] + f(u) + \mu\, g(u) + h
\end{equation}
in $W^\ast$, where $\Ap, \Bp, f, g \in C(W,W^\ast)$ are potential operators satisfying the following assumptions, $\lambda > 0$ and $\mu \in \R$ are parameters, and $h \in W^\ast \minuszero$:
\begin{enumerate}
\item[$(A_1)$] $\Ap$ is $(p - 1)$-homogeneous and odd for some $p \in (1,\infty)$: $\Ap[(tu)] = |t|^{p-2}\, t\, \Ap[u]$ for all $u \in W$ and $t \in \R$,
\item[$(A_2)$] $\dualp{\Ap[u]}{v} \le \norm{u}^{p-1} \norm{v}$ for all $u, v \in W$, and equality holds if and only if $\alpha u = \beta v$ for some $\alpha, \beta \ge 0$, not both zero (in particular, $\dualp{\Ap[u]}{u} = \norm{u}^p$ for all $u \in W$),
\item[$(B_1)$] $\Bp$ is $(p - 1)$-homogeneous and odd: $\Bp[(tu)] = |t|^{p-2}\, t\, \Bp[u]$ for all $u \in W$ and $t \in \R$,
\item[$(B_2)$] $\dualp{\Bp[u]}{u} > 0$ for all $u \in W \minuszero$, and $\dualp{\Bp[u]}{v} \le \dualp{\Bp[u]}{u}^{(p-1)/p} \dualp{\Bp[v]}{v}^{1/p}$ for all $u, v \in W$,
\item[$(B_3)$] $\Bp$ is a compact operator,
\item[$(F_1)$] the potential $F$ of $f$ with $F(0) = 0$ satisfies $F(u) = \o(\norm{u}^p)$ as $u \to 0$,
\item[$(F_2)$] $F(u) \ge 0$ for all $u \in W$,
\item[$(F_3)$] $F$ is bounded on bounded subsets of $W$,
\item[$(G)$] the potential $G$ of $g$ with $G(0) = 0$ is bounded on bounded subsets of $W$.
\end{enumerate}
Solutions of equation \eqref{1.3} coincide with critical points of the $C^1$-functional
\begin{equation} \label{1.4}
E(u) = I_p(u) - \lambda J_p(u) - F(u) - \mu\, G(u) - \dualp{h}{u}, \quad u \in W,
\end{equation}
where
\begin{equation} \label{1.5}
I_p(u) = \frac{1}{p} \dualp{\Ap[u]}{u}, \qquad J_p(u) = \frac{1}{p} \dualp{\Bp[u]}{u}
\end{equation}
are the potentials of $\Ap$ and $\Bp$ satisfying $I_p(0) = 0$ and $J_p(0) = 0$, respectively (see Perera \cite[Proposition 3.1]{MR4293883}). The nonlinear eigenvalue problem
\begin{equation} \label{1.6}
\Ap[u] = \lambda \Bp[u]
\end{equation}
will play a role in our result. Let $\M = \set{u \in W : I_p(u) = 1}$. Then $\M \subset W \minuszero$ is a bounded complete symmetric $C^1$-Finsler manifold radially homeomorphic to the unit sphere in $W$, and eigenvalues of problem \eqref{1.6} coincide with critical values of the $C^1$-functional
\[
\Psi(u) = \frac{1}{J_p(u)}, \quad u \in \M.
\]
Denote by $\F$ the class of symmetric subsets of $\M$ and by $i(M)$ the $\Z_2$-cohomological index of $M \in \F$ (see Fadell and Rabinowitz \cite{MR0478189}), let $\F_k = \set{M \in \F : i(M) \ge k}$, and set
\[
\lambda_k := \inf_{M \in \F_k}\, \sup_{u \in M}\, \Psi(u), \quad k \in \N.
\]
Then $\lambda_1 > 0$ is the first eigenvalue and $\lambda_1 \le \lambda_2 \le \cdots$ is an unbounded sequence of eigenvalues. Moreover, denoting by $\Psi^a = \set{u \in \M : \Psi(u) \le a}$ (resp. $\Psi_a = \set{u \in \M : \Psi(u) \ge a}$) the sublevel (resp. superlevel) sets of $\Psi$, if $\lambda_k < \lambda_{k+1}$, then
\begin{equation} \label{1.7}
i(\Psi^{\lambda_k}) = i(\M \setminus \Psi_{\lambda_{k+1}}) = k
\end{equation}
and $\Psi^{\lambda_k}$ has a compact symmetric subset of index $k$ (see Perera et al.\! \cite[Theorem 4.6]{MR2640827} and Perera \cite[Theorem 1.3]{MR4293883}).

We assume that there is a threshold level $c_{\mu,\,h}^\ast > 0$ such that $E$ satisfies the \PS{c} condition at all levels $c < c_{\mu,\,h}^\ast$. Set
\begin{equation} \label{1.8}
c^\ast = \liminf_{\mu,\, \dnorm{h} \to 0}\, c_{\mu,\,h}^\ast
\end{equation}
and
\begin{equation} \label{1.9}
E_0(u) = I_p(u) - \lambda J_p(u) - F(u), \quad u \in W.
\end{equation}
Let $\pi_\M : W \minuszero \to \M,\, u \mapsto u/I_p(u)^{1/p}$ be the radial projection onto $\M$. We will prove the following theorem.

\begin{theorem} \label{Theorem 1.5}
Let $\lambda_k \le \lambda < \lambda_{k+1}$. Assume that there exist $R > 0$ and, for all sufficiently small $\delta > 0$, a compact symmetric subset $C_\delta$ of $\Psi^{\lambda + \delta}$ with $i(C_\delta) = k$ and $w_\delta \in \M \setminus C_\delta$ such that, setting $A_\delta = \set{\pi_\M((1 - \tau)\, v + \tau w_\delta) : v \in C_\delta,\, 0 \le \tau \le 1}$, we have
\begin{equation} \label{1.10}
\sup_{u \in A_\delta}\, E_0(Ru) \le 0
\end{equation}
and
\begin{equation} \label{1.11}
\sup_{u \in A_\delta,\, 0 \le t \le R}\, E_0(tu) < c^\ast.
\end{equation}
Then $\exists\, \mu_0 > 0$ such that equation \eqref{1.3} has two nontrivial solutions $u_1$ and $u_2$ satisfying
\begin{equation} \label{1.12}
E(u_1) < E(u_2), \qquad 0 < E(u_2) < c_{\mu,\,h}^\ast
\end{equation}
for all $\mu \in \R$ and $h \in W^\ast \minuszero$ with $|\mu| + \dnorm{h} < \mu_0$.
\end{theorem}

Proof of this theorem and those of its applications that follow are given in Section \ref{Proofs}.

\subsection{Critical $p$-Laplacian problems}

Consider the critical $p$-Laplacian problem
\begin{equation} \label{1.13}
\left\{\begin{aligned}
- \Delta_p\, u & = \lambda\, |u|^{p-2}\, u + \mu\, |u|^{q-2}\, u + |u|^{p^\ast - 2}\, u + h(x) && \text{in } \Omega\\[10pt]
u & = 0 && \text{on } \bdry{\Omega},
\end{aligned}\right.
\end{equation}
where $\Omega$ is a bounded domain in $\R^N$, $1 < p < N$, $1 < q < p^\ast$, $p^\ast = Np/(N - p)$ is the critical Sobolev exponent, $\lambda > 0$, $\mu \in \R$, $h \in L^{{p^\ast}'}(\Omega) \minuszero$, and ${p^\ast}' = p^\ast/(p^\ast - 1)$ is the H\"{o}lder conjugate of $p^\ast$. Let
\begin{equation} \label{1.14}
E(u) = \int_\Omega \left(\frac{1}{p}\, |\nabla u|^p - \frac{\lambda}{p}\, |u|^p - \frac{\mu}{q}\, |u|^q - \frac{1}{p^\ast}\, |u|^{p^\ast} - h(x)\, u\right) dx, \quad u \in W^{1,\,p}_0(\Omega)
\end{equation}
be the associated variational functional and let
\begin{equation} \label{1.15}
S_{N,\,p} = \inf_{u \in W^{1,\,p}_0(\Omega) \setminus \set{0}}\, \frac{\dint_\Omega |\nabla u|^p\, dx}{\left(\dint_\Omega |u|^{p^\ast} dx\right)^{p/p^\ast}}
\end{equation}
be the best Sobolev constant. We have the following theorem.

\begin{theorem} \label{Theorem 1.6}
There exists $\mu_0 > 0$ such that problem \eqref{1.13} has two nontrivial solutions $u_1$ and $u_2$ satisfying
\[
E(u_1) < E(u_2), \qquad 0 < E(u_2) < \frac{1}{N}\, S_{N,\,p}^{N/p}
\]
for all $\mu \in \R$ and $h \in L^{{p^\ast}'}(\Omega) \minuszero$ with $|\mu| + \pnorm[{p^\ast}']{h} < \mu_0$ in each of the following cases:
\begin{enumroman}
\item $N \ge p^2$ and $\lambda > 0$ is not a Dirichlet eigenvalue of $- \Delta_p$ in $\Omega$,
\item $N\, (N - p^2) > p^2$ and $\lambda > 0$.
\end{enumroman}
\end{theorem}

We note that Theorem \ref{Theorem 1.6} allows the full subcritical range $1 < q < p^\ast$ for $q$ and makes no assumptions on the sign of $\mu$.

\begin{remark}
Theorem \ref{Theorem 1.1} is the special case $\mu = 0$ of Theorem \ref{Theorem 1.6}.
\end{remark}

\subsection{$p$-Laplacian problems with critical Hardy-Sobolev exponents}

Consider the problem
\begin{equation} \label{1.16}
\left\{\begin{aligned}
- \Delta_p\, u & = \lambda\, |u|^{p-2}\, u + \frac{|u|^{p^\ast(\sigma) - 2}}{|x|^\sigma}\, u + h(x) && \text{in } \Omega\\[10pt]
u & = 0 && \text{on } \bdry{\Omega},
\end{aligned}\right.
\end{equation}
where $\Omega$ is a bounded domain in $\R^N$ containing the origin, $1 < p < N$, $0 < \sigma < p$, $p^\ast(\sigma) = (N - \sigma)\, p/(N - p)$ is the critical Hardy-Sobolev exponent, $\lambda > 0$, $h \in L^{p^\ast(\sigma)'}(\Omega) \minuszero$, and $p^\ast(\sigma)' = p^\ast(\sigma)/(p^\ast(\sigma) - 1)$ is the H\"{o}lder conjugate of $p^\ast(\sigma)$. Let
\begin{equation} \label{1.17}
E(u) = \int_\Omega \left(\frac{1}{p}\, |\nabla u|^p - \frac{\lambda}{p}\, |u|^p - \frac{1}{p^\ast(\sigma)}\, \frac{|u|^{p^\ast(\sigma)}}{|x|^\sigma} - h(x)\, u\right) dx, \quad u \in W^{1,\,p}_0(\Omega)
\end{equation}
be the associated variational functional and let
\begin{equation} \label{1.18}
S_{N,\,p,\,\sigma} = \inf_{u \in W^{1,\,p}_0(\Omega) \setminus \set{0}}\, \frac{\dint_\Omega |\nabla u|^p\, dx}{\left(\dint_\Omega \frac{|u|^{p^\ast(\sigma)}}{|x|^\sigma}\, dx\right)^{p/p^\ast(\sigma)}}
\end{equation}
be the best constant in the Hardy-Sobolev inequality. We have the following theorem.

\begin{theorem} \label{Theorem 1.7}
There exists $\mu_0 > 0$ such that problem \eqref{1.16} has two nontrivial solutions $u_1$ and $u_2$ satisfying
\[
E(u_1) < E(u_2), \qquad 0 < E(u_2) < \frac{p - \sigma}{(N - \sigma)\, p}\, S_{N,\,p,\,\sigma}^{(N - \sigma)/(p - \sigma)}
\]
for all $h \in L^{p^\ast(\sigma)'}(\Omega) \minuszero$ with $\pnorm[p^\ast(\sigma)']{h} < \mu_0$ in each of the following cases:
\begin{enumroman}
\item $N \ge p^2$ and $\lambda > 0$ is not a Dirichlet eigenvalue of $- \Delta_p$ in $\Omega$,
\item $(N - \sigma)(N - p^2) > (p - \sigma)\, p$ and $\lambda > 0$.
\end{enumroman}
\end{theorem}

\begin{remark}
When $h = 0$, one nontrivial solution was obtained by Ghoussoub and Yuan \cite{MR1695021} and Perera and Zou \cite{MR3810517}.
\end{remark}

\begin{remark}
Theorem \ref{Theorem 1.1} is the special case $\sigma = 0$ of Theorem \ref{Theorem 1.7}.
\end{remark}

\subsection{Critical fractional $p$-Laplacian problems}

Consider the critical fractional $p$-Laplacian problem
\begin{equation} \label{1.19}
\left\{\begin{aligned}
(- \Delta)_p^s\, u & = \lambda\, |u|^{p-2}\, u + |u|^{p_s^\ast - 2}\, u + h(x) && \text{in } \Omega\\[10pt]
u & = 0 && \text{in } \R^N \setminus \Omega,
\end{aligned}\right.
\end{equation}
where $\Omega$ is a bounded domain in $\R^N$ with Lipschitz boundary, $(- \Delta)_p^s$ is the fractional $p$-Laplacian operator defined on smooth functions by
\[
(- \Delta)_p^s\, u(x) = 2 \lim_{\eps \searrow 0} \int_{\R^N \setminus B_\eps(x)} \frac{|u(x) - u(y)|^{p-2}\, (u(x) - u(y))}{|x - y|^{N+sp}}\, dy, \quad x \in \R^N,
\]
$s \in (0,1)$, $1 < p < N/s$, $p_s^\ast = Np/(N - sp)$ is the fractional critical Sobolev exponent, $\lambda > 0$, $h \in L^{{p_s^\ast}'}(\Omega) \minuszero$, and ${p_s^\ast}' = p_s^\ast/(p_s^\ast - 1)$ is the H\"{o}lder conjugate of $p_s^\ast$. Let
\[
[u]_{s,\,p} = \left(\int_{\R^{2N}} \frac{|u(x) - u(y)|^p}{|x - y|^{N+sp}}\, dx dy\right)^{1/p}
\]
be the Gagliardo seminorm of a measurable function $u : \R^N \to \R$ and let
\[
W^{s,\,p}(\R^N) = \set{u \in L^p(\R^N) : [u]_{s,\,p} < \infty}
\]
be the fractional Sobolev space endowed with the norm
\[
\norm[s,\,p]{u} = \left(\pnorm{u}^p + [u]_{s,\,p}^p\right)^{1/p}.
\]
We work in the closed linear subspace
\[
W^{s,\,p}_0(\Omega) = \set{u \in W^{s,\,p}(\R^N) : u = 0 \text{ a.e.\! in } \R^N \setminus \Omega},
\]
equivalently renormed by setting $\norm{\cdot} = [\cdot]_{s,\,p}$. Let
\begin{multline} \label{1.20}
E(u) = \frac{1}{p} \int_{\R^{2N}} \frac{|u(x) - u(y)|^p}{|x - y|^{N+sp}}\, dx dy - \int_\Omega \left(\frac{\lambda}{p}\, |u|^p + \frac{1}{p_s^\ast}\, |u|^{p_s^\ast} + h(x)\, u\right) dx,\\[10pt]
u \in W^{s,\,p}_0(\Omega)
\end{multline}
be the associated variational functional and let
\begin{equation} \label{1.21}
S_{N,\,p,\,s} = \inf_{u \in \dot{W}^{s,\,p}(\R^N) \setminus \set{0}}\, \frac{\dint_{\R^{2N}} \frac{|u(x) - u(y)|^p}{|x - y|^{N+sp}}\, dx dy}{\left(\dint_{\R^N} |u|^{p_s^\ast}\, dx\right)^{p/p_s^\ast}}
\end{equation}
be the best fractional Sobolev constant, where
\[
\dot{W}^{s,\,p}(\R^N) = \set{u \in L^{p_s^\ast}(\R^N) : [u]_{s,\,p} < \infty}
\]
endowed with the norm $\norm{\cdot}$. We have the following theorem.

\begin{theorem} \label{Theorem 1.9}
There exists $\mu_0 > 0$ such that problem \eqref{1.19} has two nontrivial solutions $u_1$ and $u_2$ satisfying
\[
E(u_1) < E(u_2), \qquad 0 < E(u_2) < \frac{s}{N}\, S_{N,\,p,\,s}^{N/sp}
\]
for all $h \in L^{{p_s^\ast}'}(\Omega) \minuszero$ with $\pnorm[{p_s^\ast}']{h} < \mu_0$ in each of the following cases:
\begin{enumroman}
\item $N \ge sp^2$ and $\lambda > 0$ is not a Dirichlet eigenvalue of $(- \Delta)_p^s$ in $\Omega$,
\item $N\, (N - sp^2) > s^2 p^2$ and $\lambda > 0$.
\end{enumroman}
\end{theorem}

\begin{remark}
When $h = 0$, one nontrivial solution was obtained in Mosconi et al.\! \cite{MR3530213} except when $N = sp^2$ and $\lambda > \lambda_1$ is not an eigenvalue in $(i)$, where $\lambda_1 > 0$ is the first eigenvalue.
\end{remark}

Theorem \ref{Theorem 1.9} is new even in the semilinear case $p = 2$ when $\lambda \ge \lambda_1$, which we state as the following corollary.

\begin{corollary}
There exists $\mu_0 > 0$ such that the problem
\[
\left\{\begin{aligned}
(- \Delta)^s u & = \lambda u + |u|^{2_s^\ast - 2}\, u + h(x) && \text{in } \Omega\\[10pt]
u & = 0 && \text{in } \R^N \setminus \Omega
\end{aligned}\right.
\]
has two nontrivial solutions for all $h \in L^{2N/(N+2s)}(\Omega) \minuszero$ with $\pnorm[2N/(N+2s)]{h} < \mu_0$ in each of the following cases:
\begin{enumroman}
\item $N \ge 4s$ and $\lambda > 0$ is not a Dirichlet eigenvalue of $(- \Delta)^s$ in $\Omega$,
\item $N\, (N - 4s) > 4s^2$ and $\lambda > 0$.
\end{enumroman}
\end{corollary}

\subsection{Critical $(p,q)$-Laplacian problems}

Consider the critical $(p,q)$-Laplacian problem
\begin{equation} \label{1.22}
\left\{\begin{aligned}
- \Delta_p\, u - \mu\, \Delta_q\, u & = \lambda\, |u|^{p-2}\, u + |u|^{p^\ast - 2}\, u + h(x) && \text{in } \Omega\\[10pt]
u & = 0 && \text{on } \bdry{\Omega},
\end{aligned}\right.
\end{equation}
where $\Omega$ is a bounded domain in $\R^N$, $1 < q < p < N$, $p^\ast = Np/(N - p)$, $\lambda, \mu > 0$, $h \in L^{{p^\ast}'}(\Omega) \minuszero$, and ${p^\ast}' = p^\ast/(p^\ast - 1)$. Let
\begin{equation} \label{1.23}
E(u) = \int_\Omega \left(\frac{1}{p}\, |\nabla u|^p + \frac{\mu}{q}\, |\nabla u|^q - \frac{\lambda}{p}\, |u|^p - \frac{1}{p^\ast}\, |u|^{p^\ast} - h(x)\, u\right) dx, \quad u \in W^{1,\,p}_0(\Omega)
\end{equation}
be the associated variational functional and let $S_{N,\,p}$ be as in \eqref{1.15}. We have the following theorem.

\begin{theorem} \label{Theorem 1.10}
There exists $\mu_0 > 0$ such that problem \eqref{1.22} has two nontrivial solutions $u_1$ and $u_2$ satisfying
\[
E(u_1) < E(u_2), \qquad 0 < E(u_2) < \frac{1}{N}\, S_{N,\,p}^{N/p}
\]
for all $\mu > 0$ and $h \in L^{{p^\ast}'}(\Omega) \minuszero$ with $\mu + \pnorm[{p^\ast}']{h} < \mu_0$ in each of the following cases:
\begin{enumroman}
\item $N \ge p^2$ and $\lambda > 0$ is not a Dirichlet eigenvalue of $- \Delta_p$ in $\Omega$,
\item $N\, (N - p^2) > p^2$ and $\lambda > 0$.
\end{enumroman}
\end{theorem}

\begin{remark}
Theorem \ref{Theorem 1.1} is the special case $\mu = 0$ of Theorem \ref{Theorem 1.10}.
\end{remark}

\section{Proofs} \label{Proofs}

\subsection{Proof of Theorem \ref{Theorem 1.5}}

Proof of Theorem \ref{Theorem 1.5} will be based on a special case of an abstract critical point theorem proved in Perera \cite{MR4293883}. Let $W$ be a Banach space and let $\M$ be a bounded symmetric subset of $W \minuszero$ radially homeomorphic to the unit sphere $S = \set{u \in W : \norm{u} = 1}$, i.e., the restriction to $\M$ of the radial projection $\pi : W \minuszero \to S,\, u \mapsto u/\norm{u}$ is a homeomorphism. Then the radial projection from $W \minuszero$ onto $\M$ is given by $\pi_\M = (\restr{\pi}{\M})^{-1} \comp \pi$. For a symmetric set $A \subset W \minuszero$, we denote by $i(A)$ its $\Z_2$-cohomological index (see Fadell and Rabinowitz \cite{MR0478189}). The following theorem is the special case $r = 0$ of \cite[Theorem 1.1]{MR4293883}.

\begin{theorem} \label{Theorem 2.1}
Let $E$ be a $C^1$-functional on $W$ and let $A_0$ and $B_0$ be disjoint closed symmetric subsets of $\M$ such that
\begin{equation} \label{2.1}
i(A_0) = i(\M \setminus B_0) = k < \infty.
\end{equation}
Assume that there exist $w_0 \in \M \setminus A_0$, $0 < \rho < R$, and $a < b$ such that, setting
\begin{gather}
A_1 = \set{\pi_\M((1 - \tau)\, v + \tau w_0) : v \in A_0,\, 0 \le \tau \le 1}, \notag\\[10pt]
A = \set{tv : v \in A_0,\, 0 \le t \le R} \cup \set{Ru : u \in A_1}, \qquad B = \set{\rho w : w \in B_0}, \label{2.2}\\[10pt]
A^\ast = \set{tu : u \in A_1,\, 0 \le t \le R}, \qquad B^\ast = \set{tw : w \in B_0,\, 0 \le t \le \rho}, \label{2.3}
\end{gather}
we have
\[
a < \inf_{B^\ast}\, E, \qquad \sup_A\, E < \inf_B\, E, \qquad \sup_{A^\ast}\, E < b.
\]
If $E$ satisfies the {\em \PS{c}} condition for all $c \in (a,b)$, then $E$ has two critical points $u_1$ and $u_2$ with
\[
\inf_{B^\ast}\, E \le E(u_1) \le \sup_A\, E, \qquad \inf_B\, E \le E(u_2) \le \sup_{A^\ast}\, E.
\]
\end{theorem}

We are now ready to prove Theorem \ref{Theorem 1.5}.

\begin{proof}[Proof of Theorem \ref{Theorem 1.5}]
We apply Theorem \ref{Theorem 2.1} to the functional $E$ defined in \eqref{1.4}, taking $A_0 = C_\delta$, $B_0 = \Psi_{\lambda_{k+1}}$, $w_0 = w_\delta$, and $b = c_{\mu,\,h}^\ast$, where $\delta \in (0,\lambda_{k+1} - \lambda)$ is to be chosen. Since $A_0 \subset \Psi^{\lambda + \delta}$ and $\lambda + \delta < \lambda_{k+1}$, $A_0$ and $B_0$ are disjoint. We have $i(A_0) = k$ by assumption and $i(\M \setminus B_0) = k$ by \eqref{1.7}, so \eqref{2.1} holds.

For $u \in \M$ and $t > 0$,
\begin{equation} \label{2.4}
E_0(tu) = t^p \left(1 - \frac{\lambda}{\Psi(u)}\right) - F(tu).
\end{equation}
For $w \in B_0$, this together with $(F_1)$ gives
\[
E_0(tw) \ge t^p \left(1 - \frac{\lambda}{\lambda_{k+1}} + \o(1)\right) \quad \text{as } t \to 0.
\]
Since $\lambda < \lambda_{k+1}$, it follows from this that $\exists\, \rho \in (0,R)$ such that
\begin{equation} \label{2.5}
\inf_B\, E_0 > 0,
\end{equation}
where $B$ is as in \eqref{2.2}. For $v \in A_0$ and $0 \le t \le R$, \eqref{2.4} together with $(F_2)$ gives
\begin{equation} \label{2.6}
E_0(tv) \le \frac{\delta R^p}{\lambda + \delta}
\end{equation}
since $A_0 \subset \Psi^{\lambda + \delta}$. Fix $\delta$ so small that the right-hand side is less than $\inf E_0(B)$. Then it follows from \eqref{2.6} and \eqref{1.10} that
\begin{equation} \label{2.7}
\sup_A\, E_0 \le \frac{\delta R^p}{\lambda + \delta} < \inf_B\, E_0,
\end{equation}
where $A$ is as in \eqref{2.2}.

We have
\begin{equation} \label{2.8}
|E(u) - E_0(u)| \le |\mu| |G(u)| + \dnorm{h} \norm{u} \quad \forall u \in W.
\end{equation}
Let $A^\ast$ and $B^\ast$ be as in \eqref{2.3}. Since $A$, $B$, and $A^\ast$ are bounded and $G$ is bounded on bounded sets, it follows from \eqref{2.8}, \eqref{2.5}, \eqref{2.7}, \eqref{1.8}, and \eqref{1.11} that $\exists\, \mu_0 > 0$ such that
\begin{equation} \label{2.9}
\inf_B\, E > 0, \qquad \sup_A\, E < \inf_B\, E, \qquad \sup_{A^\ast}\, E < c_{\mu,\,h}^\ast
\end{equation}
for all $\mu \in \R$ and $h \in W^\ast \minuszero$ with $|\mu| + \dnorm{h} < \mu_0$. Since $B^\ast$ is bounded and $F$ is also bounded on bounded sets,
\[
\inf_{B^\ast}\, E > - \infty.
\]
So we can apply Theorem \ref{Theorem 2.1} with $a < \inf E(B^\ast)$ to get two critical points $u_1$ and $u_2$ with
\begin{equation} \label{2.10}
\inf_{B^\ast}\, E \le E(u_1) \le \sup_A\, E, \qquad \inf_B\, E \le E(u_2) \le \sup_{A^\ast}\, E.
\end{equation}
The inequalities in \eqref{1.12} follow from \eqref{2.9} and \eqref{2.10}.
\end{proof}

\subsection{Proof of Theorem \ref{Theorem 1.6}}

We prove Theorem \ref{Theorem 1.6} by applying Theorem \ref{Theorem 1.5} with $W = W^{1,\,p}_0(\Omega)$ and the operators $\Ap, \Bp, f, g \in C(W^{1,\,p}_0(\Omega),W^{-1,\,p'}(\Omega))$ and $h \in W^{-1,\,p'}(\Omega)$ given by
\begin{multline*}
\dualp{\Ap[u]}{v} = \int_\Omega |\nabla u|^{p-2}\, \nabla u \cdot \nabla v\, dx, \quad \dualp{\Bp[u]}{v} = \int_\Omega |u|^{p-2}\, uv\, dx,\\[10pt]
\dualp{f(u)}{v} = \int_\Omega |u|^{p^\ast - 2}\, uv\, dx, \quad \dualp{g(u)}{v} = \int_\Omega |u|^{q-2}\, uv\, dx, \quad u, v \in W^{1,\,p}_0(\Omega)
\end{multline*}
and
\[
\dualp{h}{v} = \int_\Omega h(x)\, v\, dx, \quad v \in W^{1,\,p}_0(\Omega).
\]
We begin by determining a threshold level below which the functional $E$ in \eqref{1.14} satisfies the \PS{} condition.

\begin{lemma} \label{Lemma 2.2}
There exists $\kappa > 0$ such that $E$ satisfies the {\em \PS{c}} condition for all
\begin{equation} \label{2.11}
c < \frac{1}{N}\, S_{N,\,p}^{N/p} - \kappa \left(|\mu|^{(p^\ast/q)'} + \pnorm[{p^\ast}']{h}^{{p^\ast}'}\right).
\end{equation}
\end{lemma}

\begin{proof}
Let $c \in \R$ and let $\seq{u_j}$ be a sequence in $W^{1,\,p}_0(\Omega)$ such that
\begin{equation} \label{2.12}
E(u_j) = \int_\Omega \left(\frac{1}{p}\, |\nabla u_j|^p - \frac{\lambda}{p}\, |u_j|^p - \frac{\mu}{q}\, |u_j|^q - \frac{1}{p^\ast}\, |u_j|^{p^\ast} - h(x)\, u_j\right) dx = c + \o(1)
\end{equation}
and
\begin{multline} \label{2.13}
\dualp{E'(u_j)}{v} = \int_\Omega \big(|\nabla u_j|^{p-2}\, \nabla u_j \cdot \nabla v - \lambda\, |u_j|^{p-2}\, u_j\, v - \mu\, |u_j|^{q-2}\, u_j\, v - |u_j|^{p^\ast - 2}\, u_j\, v\\[10pt]
- h(x)\, v\big)\, dx = \o(\norm{v}) \quad \forall v \in W^{1,\,p}_0(\Omega).
\end{multline}
Taking $v = u_j$ in \eqref{2.13} gives
\begin{equation} \label{2.14}
\int_\Omega \left(|\nabla u_j|^p - \lambda\, |u_j|^p - \mu\, |u_j|^q - |u_j|^{p^\ast} - h(x)\, u_j\right) dx = \o(\norm{u_j}).
\end{equation}
Let $r \in (p,p^\ast)$. Dividing \eqref{2.14} by $r$ and subtracting from \eqref{2.12} gives
\begin{multline*}
\int_\Omega \bigg[\left(\frac{1}{p} - \frac{1}{r}\right) |\nabla u_j|^p - \lambda \left(\frac{1}{p} - \frac{1}{r}\right) |u_j|^p - \mu \left(\frac{1}{q} - \frac{1}{r}\right) |u_j|^q + \left(\frac{1}{r} - \frac{1}{p^\ast}\right) |u_j|^{p^\ast}\\[10pt]
- \left(1 - \frac{1}{r}\right) h(x)\, u_j\bigg] dx = c + \o(1) + \o(\norm{u_j}),
\end{multline*}
and it follows from this that $\seq{u_j}$ is bounded. So a renamed subsequence converges to some $u$ weakly in $W^{1,\,p}_0(\Omega)$, strongly in $L^t(\Omega)$ for all $t \in [1,p^\ast)$, and a.e.\! in $\Omega$. Setting $\widetilde{u}_j = u_j - u$, we will show that $\widetilde{u}_j \to 0$ in $W^{1,\,p}_0(\Omega)$.

Equation \eqref{2.14} gives
\begin{equation} \label{2.15}
\norm{u_j}^p = \pnorm[p^\ast]{u_j}^{p^\ast} + \int_\Omega \left(\lambda\, |u|^p + \mu\, |u|^q + h(x)\, u\right) dx + \o(1).
\end{equation}
Taking $v = u$ in \eqref{2.13} and passing to the limit gives
\begin{equation} \label{2.16}
\norm{u}^p = \pnorm[p^\ast]{u}^{p^\ast} + \int_\Omega \left(\lambda\, |u|^p + \mu\, |u|^q + h(x)\, u\right) dx.
\end{equation}
Since
\begin{equation} \label{2.17}
\norm{\widetilde{u}_j}^p = \norm{u_j}^p - \norm{u}^p + \o(1)
\end{equation}
and
\[
\pnorm[p^\ast]{\widetilde{u}_j}^{p^\ast} = \pnorm[p^\ast]{u_j}^{p^\ast} - \pnorm[p^\ast]{u}^{p^\ast} + \o(1)
\]
by the Br{\'e}zis-Lieb lemma \cite[Theorem 1]{MR699419}, \eqref{2.15} and \eqref{2.16} imply
\[
\norm{\widetilde{u}_j}^p = \pnorm[p^\ast]{\widetilde{u}_j}^{p^\ast} + \o(1) \le \frac{\norm{\widetilde{u}_j}^{p^\ast}}{S_{N,\,p}^{p^\ast/p}} + \o(1),
\]
so
\begin{equation} \label{2.18}
\norm{\widetilde{u}_j}^p \left(S_{N,\,p}^{N/(N-p)} - \norm{\widetilde{u}_j}^{p^2/(N-p)}\right) \le \o(1).
\end{equation}
On the other hand, \eqref{2.12} gives
\[
c = \frac{1}{p} \norm{u_j}^p - \frac{1}{p^\ast} \pnorm[p^\ast]{u_j}^{p^\ast} - \int_\Omega \left(\frac{\lambda}{p}\, |u|^p + \frac{\mu}{q}\, |u|^q + h(x)\, u\right) dx + \o(1),
\]
and a straightforward calculation combining this with \eqref{2.15}--\eqref{2.17} gives
\[
c = \frac{1}{N} \norm{\widetilde{u}_j}^p + \int_\Omega \left[\frac{1}{N}\, |u|^{p^\ast} - \mu \left(\frac{1}{q} - \frac{1}{p}\right) |u|^q - \left(1 - \frac{1}{p}\right) h(x)\, u\right] dx + \o(1).
\]
The integral on the right-hand side is greater than or equal to
\[
\frac{1}{N} \pnorm[p^\ast]{u}^{p^\ast} - |\mu| \abs{\frac{1}{q} - \frac{1}{p}} \vol{\Omega}^{1 - q/p^\ast} \pnorm[p^\ast]{u}^q - \left(1 - \frac{1}{p}\right) \pnorm[{p^\ast}']{h} \pnorm[p^\ast]{u} \ge - \kappa \left(|\mu|^{(p^\ast/q)'} + \pnorm[{p^\ast}']{h}^{{p^\ast}'}\right)
\]
for some $\kappa > 0$ by the H\"{o}lder and Young's inequalities, so
\[
\norm{\widetilde{u}_j}^p \le N \left[c + \kappa \left(|\mu|^{(p^\ast/q)'} + \pnorm[{p^\ast}']{h}^{{p^\ast}'}\right)\right] + \o(1).
\]
Combining this with \eqref{2.18} shows that $\widetilde{u}_j \to 0$ when \eqref{2.11} holds.
\end{proof}

We will apply Theorem \ref{Theorem 1.5} with
\[
c_{\mu,\,h}^\ast = \frac{1}{N}\, S_{N,\,p}^{N/p} - \kappa \left(|\mu|^{(p^\ast/q)'} + \pnorm[{p^\ast}']{h}^{{p^\ast}'}\right),
\]
where $\kappa > 0$ is as in Lemma \ref{Lemma 2.2}. Note that
\[
\lim_{\mu,\, \pnorm[{p^\ast}']{h} \to 0}\, c_{\mu,\,h}^\ast = \frac{1}{N}\, S_{N,\,p}^{N/p}.
\]
We have
\begin{gather*}
\M = \set{u \in W^{1,\,p}_0(\Omega) : \norm{u}^p = p},\\[10pt]
\Psi(u) = \frac{p}{\pnorm[p]{u}^p}, \quad u \in \M,\\[10pt]
\pi_\M(u) = \frac{p^{1/p}\, u}{\norm{u}}, \quad u \in W^{1,\,p}_0(\Omega) \minuszero,
\end{gather*}
and
\[
E_0(u) = \int_\Omega \left(\frac{1}{p}\, |\nabla u|^p - \frac{\lambda}{p}\, |u|^p - \frac{1}{p^\ast}\, |u|^{p^\ast}\right) dx, \quad u \in W^{1,\,p}_0(\Omega).
\]
Let $\lambda_k \le \lambda < \lambda_{k+1}$. We need to show that there exist $R > 0$ and, for all sufficiently small $\delta > 0$, a compact symmetric subset $C_\delta$ of $\Psi^{\lambda + \delta}$ with $i(C_\delta) = k$ and $w_\delta \in \M \setminus C_\delta$ such that, setting $A_\delta = \set{\pi_\M((1 - \tau)\, v + \tau w_\delta) : v \in C_\delta,\, 0 \le \tau \le 1}$, we have
\begin{equation} \label{2.19}
\sup_{u \in A_\delta}\, E_0(Ru) \le 0, \qquad \sup_{u \in A_\delta,\, 0 \le t \le R}\, E_0(tu) < \frac{1}{N}\, S_{N,\,p}^{N/p}.
\end{equation}

Since $\lambda_k < \lambda_{k+1}$, $\Psi^{\lambda_k}$ has a compact symmetric subset $C_0$ of index $k$ that is bounded in $L^\infty(\Omega) \cap C^1_\loc(\Omega)$ (see Degiovanni and Lancelotti \cite[Theorem 2.3]{MR2514055}). We may assume without loss of generality that $0 \in \Omega$. Let $\rho_0 = \dist{0}{\bdry{\Omega}}$, let $\eta : [0,\infty) \to [0,1]$ be a smooth function such that $\eta(t) = 0$ for $t \le 3/4$ and $\eta(t) = 1$ for $t \ge 1$, let
\[
u_\rho(x) = \eta\bigg(\frac{|x|}{\rho}\bigg)\, u(x), \quad u \in C_0,\, 0 < \rho \le \rho_0/2,
\]
and let
\[
C = \set{\pi_\M(u_\rho) : u \in C_0}.
\]

\begin{lemma} \label{Lemma 2.3}
The set $C$ is a compact symmetric subset of $\Psi^{\lambda_k + c_1\, \rho^{N-p}}$ for some constant $c_1 > 0$. If $\lambda_k + c_1\, \rho^{N-p} < \lambda_{k+1}$, then $i(C) = k$.
\end{lemma}

\begin{proof}
Let $u \in C_0$. Since functions in $C_0$ are bounded in $C^1(B_{\rho_0/2}(0))$ and belong to $\Psi^{\lambda_k}$,
\[
\int_\Omega |\nabla u_\rho|^p\, dx \le \int_{\Omega \setminus B_\rho(0)} |\nabla u|^p\, dx + \int_{B_\rho(0)} \left(|\nabla u| + \frac{|\eta'||u|}{\rho}\right)^p dx \le p + c_2\, \rho^{N-p}
\]
and
\[
\int_\Omega |u_\rho|^p\, dx \ge \int_{\Omega \setminus B_\rho(0)} |u|^p\, dx = \int_\Omega |u|^p\, dx - \int_{B_\rho(0)} |u|^p\, dx \ge \frac{p}{\lambda_k} - c_3\, \rho^N
\]
for some constants $c_2, c_3 > 0$. So
\[
\Psi(\pi_\M(u_\rho)) = \frac{\dint_\Omega |\nabla u_\rho|^p\, dx}{\dint_\Omega |u_\rho|^p\, dx} \le \lambda_k + c_1\, \rho^{N-p}
\]
for some constant $c_1 > 0$. Then $C \subset \Psi^{\lambda_k + c_1\, \rho^{N-p}}$. Since $C_0$ is a compact symmetric set and $u \mapsto \pi_\M(u_\rho)$ is an odd continuous map of $C_0$ onto $C$, $C$ is also a compact symmetric set and
\[
i(C) \ge i(C_0) = k
\]
by the monotonicity of the index. If $\lambda_k + c_1\, \rho^{N-p} < \lambda_{k+1}$, then $C \subset \M \setminus \Psi_{\lambda_{k+1}}$ and hence
\[
i(C) \le i(\M \setminus \Psi_{\lambda_{k+1}}) = k
\]
by \eqref{1.7}, so $i(C) = k$.
\end{proof}

\begin{lemma} \label{Lemma 2.4}
For any $w \in \M \setminus C$ with support in $\closure{B_{\rho/2}(0)}$, $\exists\, R > 0$ such that, setting $A = \set{\pi_\M((1 - \tau)\, v + \tau w) : v \in C,\, 0 \le \tau \le 1}$, we have
\[
\sup_{u \in A}\, E_0(Ru) \le 0.
\]
\end{lemma}

\begin{proof}
Let $u = \pi_\M((1 - \tau)\, v + \tau w) \in A$. For $R > 0$,
\[
E_0(Ru) \le \int_\Omega \left(\frac{R^p}{p}\, |\nabla u|^p - \frac{R^{p^\ast}}{p^\ast}\, |u|^{p^\ast}\right) dx = R^p - \frac{R^{p^\ast}}{p^\ast} \pnorm[p^\ast]{u}^{p^\ast},
\]
so it suffices to show that $\pnorm[p^\ast]{u}$ is bounded away from zero on $A$. By the H\"{o}lder inequality, it is enough to show that $\pnorm[p]{u}$ is bounded away from zero. Since $v, w \in \M$ have disjoint supports,
\begin{multline*}
\pnorm[p]{u}^p = \frac{p \pnorm[p]{(1 - \tau)\, v + \tau w}^p}{\norm{(1 - \tau)\, v + \tau w}^p} = \frac{p \left[(1 - \tau)^p \pnorm[p]{v}^p + \tau^p \pnorm[p]{w}^p\right]}{(1 - \tau)^p \norm{v}^p + \tau^p \norm{w}^p} = \frac{(1 - \tau)^p \pnorm[p]{v}^p + \tau^p \pnorm[p]{w}^p}{(1 - \tau)^p + \tau^p}\\[10pt]
\ge \min \set{\pnorm[p]{v}^p,\pnorm[p]{w}^p},
\end{multline*}
so it suffices to show that $\pnorm[p]{v}$ is bounded away from zero on $C$. Since $C \subset \Psi^{\lambda_k + c_1\, \rho^{N-p}}$ by Lemma \ref{Lemma 2.3}, we have
\[
\pnorm[p]{v}^p = \frac{p}{\Psi(v)} \ge \frac{p}{\lambda_k + c_1\, \rho^{N-p}}. \QED
\]
\end{proof}

Let $\delta \in (0,\lambda_{k+1} - \lambda)$, let $\rho \in (0,\rho_0/2]$ be so small that $\lambda_k + c_1\, \rho^{N-p} < \lambda + \delta$, and let $C_\delta = C$. Then $C_\delta$ is a compact symmetric subset of $\Psi^{\lambda + \delta}$ with $i(C_\delta) = k$ by Lemma \ref{Lemma 2.3}. We will show that if $\delta > 0$ is sufficiently small, then $\exists\, w_\delta \in \M \setminus C_\delta$ with support in $\closure{B_{\rho/2}(0)}$ such that, setting $A_\delta = \set{\pi_\M((1 - \tau)\, v + \tau w_\delta) : v \in C_\delta,\, 0 \le \tau \le 1}$, we have
\begin{equation} \label{2.22}
\sup_{u \in A_\delta,\, t \ge 0}\, E_0(tu) < \frac{1}{N}\, S_{N,\,p}^{N/p}.
\end{equation}
Then Lemma \ref{Lemma 2.4} will give an $R > 0$ such that \eqref{2.19} holds and complete the proof. We note that \eqref{2.22} is equivalent to
\begin{equation} \label{2.23}
\sup_{v \in C_\delta,\, t, \tau \ge 0}\, E_0(tv + \tau w_\delta) < \frac{1}{N}\, S_{N,\,p}^{N/p}.
\end{equation}

To choose $w_\delta$, recall that the infimum in \eqref{1.15} is attained by the Aubin-Talenti functions
\[
u_\eps(x) = \frac{c_{N,p}\, \eps^{(N-p)/p^2}}{\big(\eps + |x|^{p/(p-1)}\big)^{(N-p)/p}}, \quad \eps > 0
\]
when $\Omega = \R^N$, where the constant $c_{N,p} > 0$ is chosen so that
\[
\int_{\R^N} |\nabla u_\eps|^p\, dx = \int_{\R^N} u_\eps^{p^\ast} dx = S_{N,\,p}^{N/p}.
\]
Let $\zeta : [0,\infty) \to [0,1]$ be a smooth function such that $\zeta(t) = 1$ for $t \le 1/4$ and $\zeta(t) = 0$ for $t \ge 1/2$, and let
\[
u_{\eps,\rho}(x) = \zeta\bigg(\frac{|x|}{\rho}\bigg)\, u_\eps(x), \quad w_{\eps,\rho}(x) = \frac{u_{\eps,\rho}(x)}{\left(\dint_{\R^N} u_{\eps,\rho}^{p^\ast}\, dx\right)^{1/p^\ast}}, \quad 0 < \rho \le \rho_0/2.
\]
Then
\begin{equation} \label{2.24}
\int_{\R^N} w_{\eps,\rho}^{p^\ast}\, dx = 1
\end{equation}
and we have
\begin{gather}
\label{2.25} \int_{\R^N} |\nabla w_{\eps,\rho}|^p\, dx \le S_{N,\,p} + c_4\, \eps^{(N-p)/p} \rho^{-(N-p)/(p-1)},\\[10pt]
\label{2.26} \int_{\R^N} w_{\eps,\rho}^p\, dx \ge \begin{cases}
c_5\, \eps^{p-1} & \text{if } N > p^2\\[10pt]
c_5\, \eps^{p-1} \abs{\log\, (\eps \rho^{-p/(p-1)})} & \text{if } N = p^2
\end{cases}
\end{gather}
for some constants $c_4, c_5 > 0$ (see, e.g., Perera and Zou \cite{MR3810517}). Let
\[
w_\delta = \pi_\M(w_{\eps,\rho}).
\]
Since functions in $C_\delta$ have their supports in $\Omega \setminus B_{3 \rho/4}(0)$, while the support of $w_\delta$ is in $\closure{B_{\rho/2}(0)}$, $w_\delta \in \M \setminus C_\delta$. We will show that \eqref{2.23} holds if $\eps, \rho > 0$ are sufficiently small.

Inequality \eqref{2.23} is equivalent to
\begin{equation} \label{2.27}
\sup_{v \in C_\delta,\, t, \tau \ge 0}\, E_0(tv + \tau w_{\eps,\rho}) < \frac{1}{N}\, S_{N,\,p}^{N/p}.
\end{equation}
For $v \in C_\delta$ and $t, \tau \ge 0$,
\[
E_0(tv + \tau w_{\eps,\rho}) = E_0(tv) + E_0(\tau w_{\eps,\rho})
\]
since $v$ and $w_{\eps,\rho}$ have disjoint supports. So
\begin{equation} \label{2.28}
\sup_{v \in C_\delta,\, t, \tau \ge 0}\, E_0(tv + \tau w_{\eps,\rho}) = \sup_{v \in C_\delta,\, t \ge 0}\, E_0(tv) + \sup_{\tau \ge 0}\, E_0(\tau w_{\eps,\rho}).
\end{equation}

\begin{lemma} \label{Lemma 2.5}
We have
\[
\sup_{v \in C_\delta,\, t \ge 0}\, E_0(tv) \le \begin{cases}
0 & \text{if } \lambda_k + c_1\, \rho^{N-p} \le \lambda < \lambda_{k+1}\\[10pt]
c_6\, \rho^{N(N-p)/p} & \text{if } \lambda = \lambda_k,
\end{cases}
\]
where $c_1$ is as in Lemma \ref{Lemma 2.3} and $c_6 > 0$ is a constant.
\end{lemma}

\begin{proof}
For $v \in C_\delta$ and $t \ge 0$,
\[
E_0(tv) = \frac{t^p}{p} \int_\Omega \big(|\nabla v|^p - \lambda\, |v|^p\big)\, dx - \frac{t^{p^\ast}}{p^\ast} \int_\Omega |v|^{p^\ast} dx,
\]
and
\begin{equation} \label{2.20}
\frac{1}{p} \int_\Omega \big(|\nabla v|^p - \lambda\, |v|^p\big)\, dx = 1 - \frac{\lambda}{\Psi(v)} \le 1 - \frac{\lambda}{\lambda_k + c_1\, \rho^{N-p}}
\end{equation}
since $C_\delta \subset \Psi^{\lambda_k + c_1\, \rho^{N-p}}$ by Lemma \ref{Lemma 2.3}. So $E_0(tv) \le 0$ if $\lambda_k + c_1\, \rho^{N-p} \le \lambda < \lambda_{k+1}$. If $\lambda = \lambda_k$, then
\begin{equation} \label{2.21}
\frac{1}{p} \int_\Omega \big(|\nabla v|^p - \lambda\, |v|^p\big)\, dx \le \frac{c_1\, \rho^{N-p}}{\lambda_k + c_1\, \rho^{N-p}} \le c_7\, \rho^{N-p},
\end{equation}
where $c_7 = c_1/\lambda_ k > 0$, and
\[
\frac{1}{p^\ast} \int_\Omega |v|^{p^\ast} dx \ge c_8
\]
for some constant $c_8 > 0$ as in the proof of Lemma \ref{Lemma 2.4}, so
\[
E_0(tv) \le c_7\, \rho^{N-p}\, t^p - c_8\, t^{p^\ast}
\]
and maximizing the right-hand side over all $t \ge 0$ gives the desired conclusion.
\end{proof}

\begin{lemma} \label{Lemma 2.5.1}
We have
\[
\sup_{\tau \ge 0}\, E_0(\tau w_{\eps,\rho}) \le \begin{cases}
\dfrac{1}{N} \left[S_{N,\,p} + c_4\, \eps^{(N-p)/p} \rho^{-(N-p)/(p-1)} - \lambda c_5\, \eps^{p-1}\right]^{N/p} & \text{if } N > p^2\\[10pt]
\dfrac{1}{N} \left[S_{N,\,p} + c_4\, \eps^{p-1} \rho^{-p} - \lambda c_5\, \eps^{p-1} \abs{\log\, (\eps \rho^{-p/(p-1)})}\right]^{N/p} & \text{if } N = p^2.
\end{cases}
\]
\end{lemma}

\begin{proof}
We have
\[
E_0(\tau w_{\eps,\rho}) = \frac{\tau^p}{p} \int_\Omega \left(|\nabla w_{\eps,\rho}|^p - \lambda w_{\eps,\rho}^p\right) dx - \frac{\tau^{p^\ast}}{p^\ast}
\]
by \eqref{2.24}, and maximizing the right-hand side over all $\tau \ge 0$ gives
\[
\sup_{\tau \ge 0}\, E_0(\tau w_{\eps,\rho}) = \frac{1}{N} \left[\int_\Omega \left(|\nabla w_{\eps,\rho}|^p - \lambda w_{\eps,\rho}^p\right) dx\right]^{N/p},
\]
so the desired conclusion follows from \eqref{2.25} and \eqref{2.26}.
\end{proof}

We can now complete the proof of Theorem \ref{Theorem 1.6}. First suppose $N \ge p^2$ and $\lambda > \lambda_1$ is not an eigenvalue. Then $\lambda_k < \lambda < \lambda_{k+1}$ for some $k \in \N$. Let $\rho \in (0,\rho_0/2]$ be so small that $\lambda_k + c_1\, \rho^{N-p} \le \lambda$. Then \eqref{2.27} follows from \eqref{2.28}, Lemma \ref{Lemma 2.5}, and Lemma \ref{Lemma 2.5.1} for sufficiently small $\eps > 0$.

Now suppose $N\, (N - p^2) > p^2$ and $\lambda \ge \lambda_1$. Then $\lambda_k \le \lambda < \lambda_{k+1}$ for some $k \in \N$. We have already considered the case where $N > p^2$ and $\lambda_k < \lambda < \lambda_{k+1}$, so suppose $\lambda = \lambda_k$. Then
\[
\sup_{v \in C_\delta,\, t, \tau \ge 0}\, E_0(tv + \tau w_{\eps,\rho}) \le \frac{1}{N} \left[S_{N,\,p} + c_4\, \eps^{(N-p)/p} \rho^{-(N-p)/(p-1)} - \lambda c_5\, \eps^{p-1}\right]^{N/p} + c_6\, \rho^{N(N-p)/p}
\]
by \eqref{2.28}, Lemma \ref{Lemma 2.5}, and Lemma \ref{Lemma 2.5.1}. Set $\rho = \eps^\alpha$, where $\alpha > 0$ is to be chosen. Then the right-hand side is less than or equal to
\[
\frac{1}{N}\, S_{N,\,p}^{N/p} \left[1 + c_9\, \eps^{(N-p)[1/p - \alpha/(p-1)]} - c_{10}\, \eps^{p-1}\right]^{N/p} + c_6\, \eps^{\alpha N(N-p)/p}
\]
for some constants $c_9, c_{10} > 0$, so \eqref{2.27} will follow for sufficiently small $\eps > 0$ if $\alpha$ can be found so that
\[
(N - p)[1/p - \alpha/(p - 1)] > p - 1
\]
and
\[
\alpha N(N - p)/p > p - 1.
\]
This is possible if and only if
\[
(p - 1)\, p/N(N - p) < (p - 1)[1/p - (p - 1)/(N - p)],
\]
i.e.,
\[
N\, (N - p^2) > p^2. \hquad \qedsymbol
\]

\subsection{Proof of Theorem \ref{Theorem 1.7}}

We prove Theorem \ref{Theorem 1.7} by applying Theorem \ref{Theorem 1.5} with $W = W^{1,\,p}_0(\Omega)$, the operators $\Ap, \Bp, f \linebreak
\in C(W^{1,\,p}_0(\Omega),W^{-1,\,p'}(\Omega))$ and $h \in W^{-1,\,p'}(\Omega)$ given by
\begin{multline*}
\dualp{\Ap[u]}{v} = \int_\Omega |\nabla u|^{p-2}\, \nabla u \cdot \nabla v\, dx, \quad \dualp{\Bp[u]}{v} = \int_\Omega |u|^{p-2}\, uv\, dx,\\[10pt]
\dualp{f(u)}{v} = \int_\Omega \frac{|u|^{p^\ast(\sigma) - 2}}{|x|^\sigma}\, uv\, dx, \quad u, v \in W^{1,\,p}_0(\Omega)
\end{multline*}
and
\[
\dualp{h}{v} = \int_\Omega h(x)\, v\, dx, \quad v \in W^{1,\,p}_0(\Omega),
\]
and $g = 0$. The proof is similar to that of Theorem \ref{Theorem 1.6}, so we will be sketchy.

\begin{lemma} \label{Lemma 2.6}
There exists $\kappa > 0$ such that the functional $E$ in \eqref{1.17} satisfies the {\em \PS{c}} condition for all
\begin{equation} \label{2.31}
c < \frac{p - \sigma}{(N - \sigma)\, p}\, S_{N,\,p,\,\sigma}^{(N - \sigma)/(p - \sigma)} - \kappa \pnorm[p^\ast(\sigma)']{h}^{p^\ast(\sigma)'}.
\end{equation}
\end{lemma}

\begin{proof}
Let $c \in \R$ and let $\seq{u_j}$ be a sequence in $W^{1,\,p}_0(\Omega)$ such that
\begin{equation} \label{2.32}
E(u_j) = \int_\Omega \left(\frac{1}{p}\, |\nabla u_j|^p - \frac{\lambda}{p}\, |u_j|^p - \frac{1}{p^\ast(\sigma)}\, \frac{|u_j|^{p^\ast(\sigma)}}{|x|^\sigma} - h(x)\, u_j\right) dx = c + \o(1)
\end{equation}
and
\begin{multline} \label{2.33}
\dualp{E'(u_j)}{v} = \int_\Omega \left(|\nabla u_j|^{p-2}\, \nabla u_j \cdot \nabla v - \lambda\, |u_j|^{p-2}\, u_j\, v - \frac{|u_j|^{p^\ast(\sigma) - 2}}{|x|^\sigma}\, u_j\, v - h(x)\, v\right) dx\\[10pt]
= \o(\norm{v}) \quad \forall v \in W^{1,\,p}_0(\Omega).
\end{multline}
Taking $v = u_j$ in \eqref{2.33} gives
\begin{equation} \label{2.34}
\int_\Omega \left(|\nabla u_j|^p - \lambda\, |u_j|^p - \frac{|u_j|^{p^\ast(\sigma)}}{|x|^\sigma} - h(x)\, u_j\right) dx = \o(\norm{u_j}).
\end{equation}
Let $r \in (p,p^\ast(\sigma))$. Dividing \eqref{2.34} by $r$ and subtracting from \eqref{2.32} gives
\begin{multline*}
\hspace{-6.64pt} \int_\Omega \left[\left(\frac{1}{p} - \frac{1}{r}\right) |\nabla u_j|^p - \lambda \left(\frac{1}{p} - \frac{1}{r}\right) |u_j|^p + \left(\frac{1}{r} - \frac{1}{p^\ast(\sigma)}\right) \frac{|u_j|^{p^\ast(\sigma)}}{|x|^\sigma} - \left(1 - \frac{1}{r}\right) h(x)\, u_j\right] dx\\[10pt]
= c + \o(1) + \o(\norm{u_j}),
\end{multline*}
and it follows from this that $\seq{u_j}$ is bounded. So a renamed subsequence converges to some $u$ weakly in $W^{1,\,p}_0(\Omega)$, strongly in $L^t(\Omega)$ for all $t \in [1,p^\ast)$, and a.e.\! in $\Omega$. Setting $\widetilde{u}_j = u_j - u$, we will show that $\widetilde{u}_j \to 0$ in $W^{1,\,p}_0(\Omega)$.

Equation \eqref{2.34} gives
\begin{equation} \label{2.35}
\norm{u_j}^p = \int_\Omega \frac{|u_j|^{p^\ast(\sigma)}}{|x|^\sigma}\, dx + \int_\Omega \left(\lambda\, |u|^p + h(x)\, u\right) dx + \o(1).
\end{equation}
Taking $v = u$ in \eqref{2.33} and passing to the limit gives
\begin{equation} \label{2.36}
\norm{u}^p = \int_\Omega \frac{|u|^{p^\ast(\sigma)}}{|x|^\sigma}\, dx + \int_\Omega \left(\lambda\, |u|^p + h(x)\, u\right) dx.
\end{equation}
Since
\begin{equation} \label{2.37}
\norm{\widetilde{u}_j}^p = \norm{u_j}^p - \norm{u}^p + \o(1)
\end{equation}
by the Br{\'e}zis-Lieb lemma \cite[Theorem 1]{MR699419} and
\[
\int_\Omega \frac{|\widetilde{u}_j|^{p^\ast(\sigma)}}{|x|^\sigma}\, dx = \int_\Omega \frac{|u_j|^{p^\ast(\sigma)}}{|x|^\sigma}\, dx - \int_\Omega \frac{|u|^{p^\ast(\sigma)}}{|x|^\sigma}\, dx + \o(1)
\]
by Ghoussoub and Yuan \cite[Lemma 4.3]{MR1695021}, \eqref{2.35} and \eqref{2.36} imply
\[
\norm{\widetilde{u}_j}^p = \int_\Omega \frac{|\widetilde{u}_j|^{p^\ast(\sigma)}}{|x|^\sigma}\, dx + \o(1) \le \frac{\norm{\widetilde{u}_j}^{p^\ast(\sigma)}}{S_{N,\,p,\,\sigma}^{p^\ast(\sigma)/p}} + \o(1),
\]
so
\begin{equation} \label{2.38}
\norm{\widetilde{u}_j}^p \left(S_{N,\,p,\,\sigma}^{(N - \sigma)/(N-p)} - \norm{\widetilde{u}_j}^{(p - \sigma)\,p/(N-p)}\right) \le \o(1).
\end{equation}
On the other hand, \eqref{2.32} gives
\[
c = \frac{1}{p} \norm{u_j}^p - \frac{1}{p^\ast(\sigma)} \int_\Omega \frac{|u_j|^{p^\ast(\sigma)}}{|x|^\sigma}\, dx - \int_\Omega \left(\frac{\lambda}{p}\, |u|^p + h(x)\, u\right) dx + \o(1),
\]
and a straightforward calculation combining this with \eqref{2.35}--\eqref{2.37} gives
\[
c = \frac{p - \sigma}{(N - \sigma)\, p} \norm{\widetilde{u}_j}^p + \int_\Omega \left[\frac{p - \sigma}{(N - \sigma)\, p}\, \frac{|u|^{p^\ast(\sigma)}}{|x|^\sigma} - \left(1 - \frac{1}{p}\right) h(x)\, u\right] dx + \o(1).
\]
The integral on the right-hand side is greater than or equal to
\begin{multline*}
\frac{p - \sigma}{(N - \sigma)\, p} \int_\Omega \frac{|u|^{p^\ast(\sigma)}}{|x|^\sigma}\, dx - \left(1 - \frac{1}{p}\right) \left(\int_\Omega |x|^{\sigma\, p^\ast(\sigma)'/p^\ast(\sigma)}\, |h(x)|^{p^\ast(\sigma)'}\, dx\right)^{1/p^\ast(\sigma)'}\\[10pt]
\left(\int_\Omega \frac{|u|^{p^\ast(\sigma)}}{|x|^\sigma}\, dx\right)^{1/p^\ast(\sigma)} \ge - \kappa \pnorm[p^\ast(\sigma)']{h}^{p^\ast(\sigma)'}
\end{multline*}
for some $\kappa > 0$ by the H\"{o}lder and Young's inequalities, so
\[
\norm{\widetilde{u}_j}^p \le \frac{(N - \sigma)\, p}{p - \sigma} \left(c + \kappa \pnorm[p^\ast(\sigma)']{h}^{p^\ast(\sigma)'}\right) + \o(1).
\]
Combining this with \eqref{2.38} shows that $\widetilde{u}_j \to 0$ when \eqref{2.31} holds.
\end{proof}

We will apply Theorem \ref{Theorem 1.5} with
\[
c_{\mu,\,h}^\ast = \frac{p - \sigma}{(N - \sigma)\, p}\, S_{N,\,p,\,\sigma}^{(N - \sigma)/(p - \sigma)} - \kappa \pnorm[p^\ast(\sigma)']{h}^{p^\ast(\sigma)'},
\]
where $\kappa > 0$ is as in Lemma \ref{Lemma 2.6}, noting that
\[
\lim_{\pnorm[p^\ast(\sigma)']{h} \to 0}\, c_{\mu,\,h}^\ast = \frac{p - \sigma}{(N - \sigma)\, p}\, S_{N,\,p,\,\sigma}^{(N - \sigma)/(p - \sigma)}.
\]
We have
\begin{gather*}
\M = \set{u \in W^{1,\,p}_0(\Omega) : \norm{u}^p = p},\\[10pt]
\Psi(u) = \frac{p}{\pnorm[p]{u}^p}, \quad u \in \M,\\[10pt]
\pi_\M(u) = \frac{p^{1/p}\, u}{\norm{u}}, \quad u \in W^{1,\,p}_0(\Omega) \minuszero
\end{gather*}
as in the proof of Theorem \ref{Theorem 1.6} and
\[
E_0(u) = \int_\Omega \left(\frac{1}{p}\, |\nabla u|^p - \frac{\lambda}{p}\, |u|^p - \frac{1}{p^\ast(\sigma)}\, \frac{|u|^{p^\ast(\sigma)}}{|x|^\sigma}\right) dx, \quad u \in W^{1,\,p}_0(\Omega).
\]
Let $\lambda_k \le \lambda < \lambda_{k+1}$. We need to show that there exist $R > 0$ and, for all sufficiently small $\delta > 0$, a compact symmetric subset $C_\delta$ of $\Psi^{\lambda + \delta}$ with $i(C_\delta) = k$ and $w_\delta \in \M \setminus C_\delta$ such that, setting $A_\delta = \set{\pi_\M((1 - \tau)\, v + \tau w_\delta) : v \in C_\delta,\, 0 \le \tau \le 1}$, we have
\begin{equation} \label{2.39}
\sup_{u \in A_\delta}\, E_0(Ru) \le 0, \qquad \sup_{u \in A_\delta,\, 0 \le t \le R}\, E_0(tu) < \frac{p - \sigma}{(N - \sigma)\, p}\, S_{N,\,p,\,\sigma}^{(N - \sigma)/(p - \sigma)}.
\end{equation}
Let $\rho_0 = \dist{0}{\bdry{\Omega}}$, let $0 < \rho \le \rho_0/2$, and let $C$ be as in the proof of Theorem \ref{Theorem 1.6}.

\begin{lemma} \label{Lemma 2.7}
For any $w \in \M \setminus C$ with support in $\closure{B_{\rho/2}(0)}$, $\exists\, R > 0$ such that, setting $A = \set{\pi_\M((1 - \tau)\, v + \tau w) : v \in C,\, 0 \le \tau \le 1}$, we have
\[
\sup_{u \in A}\, E_0(Ru) \le 0.
\]
\end{lemma}

\begin{proof}
Let $u = \pi_\M((1 - \tau)\, v + \tau w) \in A$. For $R > 0$,
\[
E_0(Ru) \le \int_\Omega \left(\frac{R^p}{p}\, |\nabla u|^p - \frac{R^{p^\ast(\sigma)}}{p^\ast(\sigma)}\, \frac{|u|^{p^\ast(\sigma)}}{|x|^\sigma}\right) dx = R^p - \frac{R^{p^\ast(\sigma)}}{p^\ast(\sigma)} \int_\Omega \frac{|u|^{p^\ast(\sigma)}}{|x|^\sigma}\, dx,
\]
so it suffices to show that the last integral is bounded away from zero on $A$. By the H\"{o}lder inequality,
\[
\int_\Omega |u|^p\, dx \le \left(\int_\Omega |x|^{\sigma\, p/(p^\ast(\sigma) - p)}\, dx\right)^{1 - p/p^\ast(\sigma)} \left(\int_\Omega \frac{|u|^{p^\ast(\sigma)}}{|x|^\sigma}\, dx\right)^{p/p^\ast(\sigma)},
\]
and $\pnorm[p]{u}$ is bounded away from zero as in the proof of Lemma \ref{Lemma 2.4}, so the desired conclusion follows.
\end{proof}

Let $\delta \in (0,\lambda_{k+1} - \lambda)$, let $\rho \in (0,\rho_0/2]$ be so small that $\lambda_k + c_1\, \rho^{N-p} < \lambda + \delta$, and let $C_\delta = C$. Then $C_\delta$ is a compact symmetric subset of $\Psi^{\lambda + \delta}$ with $i(C_\delta) = k$ by Lemma \ref{Lemma 2.3}. We will show that if $\delta > 0$ is sufficiently small, then $\exists\, w_\delta \in \M \setminus C_\delta$ with support in $\closure{B_{\rho/2}(0)}$ such that, setting $A_\delta = \set{\pi_\M((1 - \tau)\, v + \tau w_\delta) : v \in C_\delta,\, 0 \le \tau \le 1}$, we have
\begin{equation} \label{2.40}
\sup_{u \in A_\delta,\, t \ge 0}\, E_0(tu) < \frac{p - \sigma}{(N - \sigma)\, p}\, S_{N,\,p,\,\sigma}^{(N - \sigma)/(p - \sigma)}.
\end{equation}
Then Lemma \ref{Lemma 2.7} will give an $R > 0$ such that \eqref{2.39} holds and complete the proof. We note that \eqref{2.40} is equivalent to
\begin{equation} \label{2.41}
\sup_{v \in C_\delta,\, t, \tau \ge 0}\, E_0(tv + \tau w_\delta) < \frac{p - \sigma}{(N - \sigma)\, p}\, S_{N,\,p,\,\sigma}^{(N - \sigma)/(p - \sigma)}.
\end{equation}

To choose $w_\delta$, recall that the infimum in \eqref{1.18} is attained by the family of functions
\[
u_\eps(x) = \frac{c_{N,p,\sigma}\, \eps^{(N-p)/(p - \sigma)\,p}}{\big(\eps + |x|^{(p - \sigma)/(p-1)}\big)^{(N-p)/(p - \sigma)}}, \quad \eps > 0
\]
when $\Omega = \R^N$, where the constant $c_{N,p,\sigma} > 0$ is chosen so that
\[
\int_{\R^N} |\nabla u_\eps|^p\, dx = \int_{\R^N} \frac{u_\eps^{p^\ast(\sigma)}}{|x|^\sigma}\, dx = S_{N,\,p,\,\sigma}^{(N - \sigma)/(p - \sigma)}
\]
(see \cite[Theorem 3.1.(2)]{MR1695021}). Let $\zeta : [0,\infty) \to [0,1]$ be a smooth function such that $\zeta(t) = 1$ for $t \le 1/4$ and $\zeta(t) = 0$ for $t \ge 1/2$, and let
\[
u_{\eps,\rho}(x) = \zeta\bigg(\frac{|x|}{\rho}\bigg)\, u_\eps(x), \quad w_{\eps,\rho}(x) = \frac{u_{\eps,\rho}(x)}{\left(\dint_{\R^N} \frac{u_{\eps,\rho}^{p^\ast(\sigma)}}{|x|^\sigma}\, dx\right)^{1/p^\ast(\sigma)}}, \quad 0 < \rho \le \rho_0/2.
\]
Then
\begin{equation} \label{2.42}
\int_{\R^N} \frac{w_{\eps,\rho}^{p^\ast(\sigma)}}{|x|^\sigma}\, dx = 1
\end{equation}
and we have
\begin{gather}
\label{2.43} \int_{\R^N} |\nabla w_{\eps,\rho}|^p\, dx \le S_{N,\,p,\,\sigma} + c_{11}\, \eps^{(N-p)/(p - \sigma)} \rho^{-(N-p)/(p-1)},\\[10pt]
\label{2.44} \int_{\R^N} w_{\eps,\rho}^p\, dx \ge \begin{cases}
c_{12}\, \eps^{(p-1)\,p/(p - \sigma)} & \text{if } N > p^2\\[10pt]
c_{12}\, \eps^{(p-1)\,p/(p - \sigma)} \abs{\log\, (\eps \rho^{-(p - \sigma)/(p-1)})} & \text{if } N = p^2
\end{cases}
\end{gather}
for some constants $c_{11}, c_{12} > 0$ (see Perera and Zou \cite{MR3810517}). Let
\[
w_\delta = \pi_\M(w_{\eps,\rho}).
\]
Since functions in $C_\delta$ have their supports in $\Omega \setminus B_{3 \rho/4}(0)$, while the support of $w_\delta$ is in $\closure{B_{\rho/2}(0)}$, $w_\delta \in \M \setminus C_\delta$. We will show that \eqref{2.41} holds if $\eps, \rho > 0$ are sufficiently small.

Inequality \eqref{2.41} is equivalent to
\begin{equation} \label{2.45}
\sup_{v \in C_\delta,\, t, \tau \ge 0}\, E_0(tv + \tau w_{\eps,\rho}) < \frac{p - \sigma}{(N - \sigma)\, p}\, S_{N,\,p,\,\sigma}^{(N - \sigma)/(p - \sigma)}.
\end{equation}
For $v \in C_\delta$ and $t, \tau \ge 0$,
\[
E_0(tv + \tau w_{\eps,\rho}) = E_0(tv) + E_0(\tau w_{\eps,\rho})
\]
since $v$ and $w_{\eps,\rho}$ have disjoint supports. So
\begin{equation} \label{2.46}
\sup_{v \in C_\delta,\, t, \tau \ge 0}\, E_0(tv + \tau w_{\eps,\rho}) = \sup_{v \in C_\delta,\, t \ge 0}\, E_0(tv) + \sup_{\tau \ge 0}\, E_0(\tau w_{\eps,\rho}).
\end{equation}

\begin{lemma} \label{Lemma 2.8}
We have
\[
\sup_{v \in C_\delta,\, t \ge 0}\, E_0(tv) \le \begin{cases}
0 & \text{if } \lambda_k + c_1\, \rho^{N-p} \le \lambda < \lambda_{k+1}\\[10pt]
c_{13}\, \rho^{(N - \sigma)(N-p)/(p - \sigma)} & \text{if } \lambda = \lambda_k,
\end{cases}
\]
where $c_1$ is as in Lemma \ref{Lemma 2.3} and $c_{13} > 0$ is a constant.
\end{lemma}

\begin{proof}
For $v \in C_\delta$ and $t \ge 0$,
\[
E_0(tv) = \frac{t^p}{p} \int_\Omega \big(|\nabla v|^p - \lambda\, |v|^p\big)\, dx - \frac{t^{p^\ast(\sigma)}}{p^\ast(\sigma)} \int_\Omega \frac{|v|^{p^\ast(\sigma)}}{|x|^\sigma}\, dx
\]
and \eqref{2.20} holds. So $E_0(tv) \le 0$ if $\lambda_k + c_1\, \rho^{N-p} \le \lambda < \lambda_{k+1}$. If $\lambda = \lambda_k$, then \eqref{2.21} holds and
\[
\frac{1}{p^\ast(\sigma)} \int_\Omega \frac{|v|^{p^\ast(\sigma)}}{|x|^\sigma}\, dx \ge c_{14}
\]
for some constant $c_{14} > 0$ as in the proof of Lemma \ref{Lemma 2.7}, so
\[
E_0(tv) \le c_7\, \rho^{N-p}\, t^p - c_{14}\, t^{p^\ast(\sigma)}
\]
and maximizing the right-hand side over all $t \ge 0$ gives the desired conclusion.
\end{proof}

\begin{lemma} \label{Lemma 2.8.1}
We have
\[
\sup_{\tau \ge 0}\, E_0(\tau w_{\eps,\rho}) \le \left\{\begin{aligned}
& \frac{p - \sigma}{(N - \sigma)\, p}\, \big[S_{N,\,p,\,\sigma} + c_{11}\, \eps^{(N-p)/(p - \sigma)} \rho^{-(N-p)/(p-1)} &&\\[5pt]
& \hquad - \lambda c_{12}\, \eps^{(p-1)\,p/(p - \sigma)}\big]^{(N - \sigma)/(p - \sigma)} && \text{if } N > p^2\\[10pt]
& \frac{p - \sigma}{(N - \sigma)\, p}\, \big[S_{N,\,p,\,\sigma} + c_{11}\, \eps^{(p-1)\,p/(p - \sigma)} \rho^{-p} &&\\[5pt]
& \hquad - \lambda c_{12}\, \eps^{(p-1)\,p/(p - \sigma)} \abs{\log\, (\eps \rho^{-(p - \sigma)/(p-1)})}\big]^{(N - \sigma)/(p - \sigma)} && \text{if } N = p^2.
\end{aligned}\right.
\]
\end{lemma}

\begin{proof}
We have
\[
E_0(\tau w_{\eps,\rho}) = \frac{\tau^p}{p} \int_\Omega \left(|\nabla w_{\eps,\rho}|^p - \lambda w_{\eps,\rho}^p\right) dx - \frac{\tau^{p^\ast(\sigma)}}{p^\ast(\sigma)}
\]
by \eqref{2.42}, and maximizing the right-hand side over all $\tau \ge 0$ gives
\[
\sup_{\tau \ge 0}\, E_0(\tau w_{\eps,\rho}) = \frac{p - \sigma}{(N - \sigma)\, p} \left[\int_\Omega \left(|\nabla w_{\eps,\rho}|^p - \lambda w_{\eps,\rho}^p\right) dx\right]^{(N - \sigma)/(p - \sigma)},
\]
so the desired conclusion follows from \eqref{2.43} and \eqref{2.44}.
\end{proof}

We can now complete the proof of Theorem \ref{Theorem 1.7}. First suppose $N \ge p^2$ and $\lambda > \lambda_1$ is not an eigenvalue. Then $\lambda_k < \lambda < \lambda_{k+1}$ for some $k \in \N$. Let $\rho \in (0,\rho_0/2]$ be so small that $\lambda_k + c_1\, \rho^{N-p} \le \lambda$. Then \eqref{2.45} follows from \eqref{2.46}, Lemma \ref{Lemma 2.8}, and Lemma \ref{Lemma 2.8.1} for sufficiently small $\eps > 0$.

Now suppose $(N - \sigma)(N - p^2) > (p - \sigma)\, p$ and $\lambda \ge \lambda_1$. Then $\lambda_k \le \lambda < \lambda_{k+1}$ for some $k \in \N$. We have already considered the case where $N > p^2$ and $\lambda_k < \lambda < \lambda_{k+1}$, so suppose $\lambda = \lambda_k$. Then
\begin{multline*}
\sup_{v \in C_\delta,\, t, \tau \ge 0}\, E_0(tv + \tau w_{\eps,\rho}) \le \frac{p - \sigma}{(N - \sigma)\, p}\, \big[S_{N,\,p,\,\sigma} + c_{11}\, \eps^{(N-p)/(p - \sigma)} \rho^{-(N-p)/(p-1)}\\[10pt]
- \lambda c_{12}\, \eps^{(p-1)\,p/(p - \sigma)}\big]^{(N - \sigma)/(p - \sigma)} + c_{13}\, \rho^{(N - \sigma)(N-p)/(p - \sigma)}
\end{multline*}
by \eqref{2.46}, Lemma \ref{Lemma 2.8}, and Lemma \ref{Lemma 2.8.1}. Set $\rho = \eps^\alpha$, where $\alpha > 0$ is to be chosen. Then the right-hand side is less than or equal to
\begin{multline*}
\frac{p - \sigma}{(N - \sigma)\, p}\, S_{N,\,p,\,\sigma}^{(N - \sigma)/(p - \sigma)} \left[1 + c_{15}\, \eps^{(N-p)[1/(p - \sigma) - \alpha/(p-1)]} - c_{16}\, \eps^{(p-1)\,p/(p - \sigma)}\right]^{(N - \sigma)/(p - \sigma)}\\[10pt]
+ c_{13}\, \eps^{\alpha\,(N - \sigma)(N-p)/(p - \sigma)}
\end{multline*}
for some constants $c_{15}, c_{16} > 0$, so \eqref{2.45} will follow for sufficiently small $\eps > 0$ if $\alpha$ can be found so that
\[
(N - p)[1/(p - \sigma) - \alpha/(p - 1)] > (p - 1)\, p/(p - \sigma)
\]
and
\[
\alpha\, (N - \sigma)(N - p)/(p - \sigma) > (p - 1)\, p/(p - \sigma).
\]
This is possible if and only if
\[
(p - 1)\, p/(N - \sigma)(N - p) < (p - 1)[1 - (p - 1)\, p/(N - p)]/(p - \sigma),
\]
i.e.,
\[
(N - \sigma)(N - p^2) > (p - \sigma)\, p. \hquad \qedsymbol
\]

\subsection{Proof of Theorem \ref{Theorem 1.9}}

We prove Theorem \ref{Theorem 1.9} by applying Theorem \ref{Theorem 1.5} with $W = W^{s,\,p}_0(\Omega)$, the operators $\Ap, \Bp, f \linebreak
\in C(W^{s,\,p}_0(\Omega),W^{s,\,p}_0(\Omega)^\ast)$ and $h \in W^{s,\,p}_0(\Omega)^\ast$ given by
\begin{multline*}
\dualp{\Ap[u]}{v} = \int_{\R^{2N}} \frac{|u(x) - u(y)|^{p-2}\, (u(x) - u(y))\, (v(x) - v(y))}{|x - y|^{N+sp}}\, dx dy,\\[10pt]
\dualp{\Bp[u]}{v} = \int_\Omega |u|^{p-2}\, uv\, dx, \quad \dualp{f(u)}{v} = \int_\Omega |u|^{p_s^\ast - 2}\, uv\, dx, \quad u, v \in W^{s,\,p}_0(\Omega)
\end{multline*}
and
\[
\dualp{h}{v} = \int_\Omega h(x)\, v\, dx, \quad v \in W^{s,\,p}_0(\Omega),
\]
and $g = 0$.

\begin{lemma} \label{Lemma 2.9}
There exists $\kappa > 0$ such that the functional $E$ in \eqref{1.20} satisfies the {\em \PS{c}} condition for all
\begin{equation} \label{2.49}
c < \frac{s}{N}\, S_{N,\,p,\,s}^{N/sp} - \kappa \pnorm[{p_s^\ast}']{h}^{{p_s^\ast}'}.
\end{equation}
\end{lemma}

\begin{proof}
Let $c \in \R$ and let $\seq{u_j}$ be a sequence in $W^{s,\,p}_0(\Omega)$ such that
\begin{multline} \label{2.50}
E(u_j) = \frac{1}{p} \int_{\R^{2N}} \frac{|u_j(x) - u_j(y)|^p}{|x - y|^{N+sp}}\, dx dy - \int_\Omega \left(\frac{\lambda}{p}\, |u_j|^p + \frac{1}{p_s^\ast}\, |u_j|^{p_s^\ast} + h(x)\, u_j\right) dx\\[10pt]
= c + \o(1)
\end{multline}
and
\begin{multline} \label{2.51}
\dualp{E'(u_j)}{v} = \int_{\R^{2N}} \frac{|u_j(x) - u_j(y)|^{p-2}\, (u_j(x) - u_j(y))\, (v(x) - v(y))}{|x - y|^{N+sp}}\, dx dy\\[10pt]
- \int_\Omega \left(\lambda\, |u_j|^{p-2}\, u_j\, v + |u_j|^{p_s^\ast - 2}\, u_j\, v + h(x)\, v\right) dx = \o(\norm{v}) \quad \forall v \in W^{s,\,p}_0(\Omega).
\end{multline}
Taking $v = u_j$ in \eqref{2.51} gives
\begin{equation} \label{2.52}
\int_{\R^{2N}} \frac{|u_j(x) - u_j(y)|^p}{|x - y|^{N+sp}}\, dx dy - \int_\Omega \left(\lambda\, |u_j|^p + |u_j|^{p_s^\ast} + h(x)\, u_j\right) dx = \o(\norm{u_j}).
\end{equation}
Let $r \in (p,p_s^\ast)$. Dividing \eqref{2.52} by $r$ and subtracting from \eqref{2.50} gives
\begin{multline*}
\left(\frac{1}{p} - \frac{1}{r}\right) \int_{\R^{2N}} \frac{|u_j(x) - u_j(y)|^p}{|x - y|^{N+sp}}\, dx dy - \int_\Omega \bigg[\lambda \left(\frac{1}{p} - \frac{1}{r}\right) |u_j|^p - \left(\frac{1}{r} - \frac{1}{p_s^\ast}\right) |u_j|^{p_s^\ast}\\[10pt]
+ \left(1 - \frac{1}{r}\right) h(x)\, u_j\bigg] dx = c + \o(1) + \o(\norm{u_j}),
\end{multline*}
and it follows from this that $\seq{u_j}$ is bounded. So a renamed subsequence converges to some $u$ weakly in $W^{s,\,p}_0(\Omega)$, strongly in $L^t(\Omega)$ for all $t \in [1,p_s^\ast)$, and a.e.\! in $\Omega$. Setting $\widetilde{u}_j = u_j - u$, we will show that $\widetilde{u}_j \to 0$ in $W^{s,\,p}_0(\Omega)$.

Equation \eqref{2.52} gives
\begin{equation} \label{2.53}
\norm{u_j}^p = \pnorm[p_s^\ast]{u_j}^{p_s^\ast} + \int_\Omega \left(\lambda\, |u|^p + h(x)\, u\right) dx + \o(1).
\end{equation}
Taking $v = u$ in \eqref{2.51} and passing to the limit gives
\begin{equation} \label{2.54}
\norm{u}^p = \pnorm[p_s^\ast]{u}^{p_s^\ast} + \int_\Omega \left(\lambda\, |u|^p + h(x)\, u\right) dx.
\end{equation}
Since
\begin{equation} \label{2.55}
\norm{\widetilde{u}_j}^p = \norm{u_j}^p - \norm{u}^p + \o(1)
\end{equation}
by Perera et al.\! \cite[Lemma 3.2]{MR3458311} and
\[
\pnorm[p_s^\ast]{\widetilde{u}_j}^{p_s^\ast} = \pnorm[p_s^\ast]{u_j}^{p_s^\ast} - \pnorm[p_s^\ast]{u}^{p_s^\ast} + \o(1)
\]
by the Br{\'e}zis-Lieb lemma \cite[Theorem 1]{MR699419}, \eqref{2.53} and \eqref{2.54} imply
\[
\norm{\widetilde{u}_j}^p = \pnorm[p_s^\ast]{\widetilde{u}_j}^{p_s^\ast} + \o(1) \le \frac{\norm{\widetilde{u}_j}^{p_s^\ast}}{S_{N,\,p,\,s}^{p_s^\ast/p}} + \o(1),
\]
so
\begin{equation} \label{2.56}
\norm{\widetilde{u}_j}^p \left(S_{N,\,p,\,s}^{N/(N-sp)} - \norm{\widetilde{u}_j}^{sp^2/(N-sp)}\right) \le \o(1).
\end{equation}
On the other hand, \eqref{2.50} gives
\[
c = \frac{1}{p} \norm{u_j}^p - \frac{1}{p_s^\ast} \pnorm[p_s^\ast]{u_j}^{p_s^\ast} - \int_\Omega \left(\frac{\lambda}{p}\, |u|^p + h(x)\, u\right) dx + \o(1),
\]
and a straightforward calculation combining this with \eqref{2.53}--\eqref{2.55} gives
\[
c = \frac{s}{N} \norm{\widetilde{u}_j}^p + \int_\Omega \left[\frac{s}{N}\, |u|^{p_s^\ast} - \left(1 - \frac{1}{p}\right) h(x)\, u\right] dx + \o(1).
\]
The integral on the right-hand side is greater than or equal to
\[
\frac{s}{N} \pnorm[p_s^\ast]{u}^{p_s^\ast} - \left(1 - \frac{1}{p}\right) \pnorm[{p_s^\ast}']{h} \pnorm[p_s^\ast]{u} \ge - \kappa \pnorm[{p_s^\ast}']{h}^{{p_s^\ast}'}
\]
for some $\kappa > 0$ by the H\"{o}lder and Young's inequalities, so
\[
\norm{\widetilde{u}_j}^p \le \frac{N}{s} \left(c + \kappa \pnorm[{p_s^\ast}']{h}^{{p_s^\ast}'}\right) + \o(1).
\]
Combining this with \eqref{2.56} shows that $\widetilde{u}_j \to 0$ when \eqref{2.49} holds.
\end{proof}

We will apply Theorem \ref{Theorem 1.5} with
\[
c_{\mu,\,h}^\ast = \frac{s}{N}\, S_{N,\,p,\,s}^{N/sp} - \kappa \pnorm[{p_s^\ast}']{h}^{{p_s^\ast}'},
\]
where $\kappa > 0$ is as in Lemma \ref{Lemma 2.9}, noting that
\[
\lim_{\pnorm[{p_s^\ast}']{h} \to 0}\, c_{\mu,\,h}^\ast = \frac{s}{N}\, S_{N,\,p,\,s}^{N/sp}.
\]
We have
\begin{gather*}
\M = \set{u \in W^{s,\,p}_0(\Omega) : \norm{u}^p = p},\\[10pt]
\Psi(u) = \frac{p}{\pnorm[p]{u}^p}, \quad u \in \M,\\[10pt]
\pi_\M(u) = \frac{p^{1/p}\, u}{\norm{u}}, \quad u \in W^{s,\,p}_0(\Omega) \minuszero,
\end{gather*}
and
\[
E_0(u) = \frac{1}{p} \int_{\R^{2N}} \frac{|u(x) - u(y)|^p}{|x - y|^{N+sp}}\, dx dy - \int_\Omega \left(\frac{\lambda}{p}\, |u|^p + \frac{1}{p_s^\ast}\, |u|^{p_s^\ast}\right) dx, \quad u \in W^{s,\,p}_0(\Omega).
\]
Let $\lambda_k \le \lambda < \lambda_{k+1}$. We need to show that there exist $R > 0$ and, for all sufficiently small $\delta > 0$, a compact symmetric subset $C_\delta$ of $\Psi^{\lambda + \delta}$ with $i(C_\delta) = k$ and $w_\delta \in \M \setminus C_\delta$ such that, setting $A_\delta = \set{\pi_\M((1 - \tau)\, v + \tau w_\delta) : v \in C_\delta,\, 0 \le \tau \le 1}$, we have
\begin{equation} \label{2.57}
\sup_{u \in A_\delta}\, E_0(Ru) \le 0, \qquad \sup_{u \in A_\delta,\, 0 \le t \le R}\, E_0(tu) < \frac{s}{N}\, S_{N,\,p,\,s}^{N/sp}.
\end{equation}

Since $\lambda_k < \lambda_{k+1}$, $\Psi^{\lambda_k}$ has a compact symmetric subset $C_0$ of index $k$ that is bounded in $L^\infty(\Omega)$ (see Mosconi et al.\! \cite[Proposition 3.1]{MR3530213}). We may assume without loss of generality that $0 \in \Omega$. Let $\rho_0 = \dist{0}{\bdry{\Omega}}$, let $\eta : [0,\infty) \to [0,1]$ be a smooth function such that $\eta(t) = 0$ for $t \le 3/4$ and $\eta(t) = 1$ for $t \ge 1$, let
\[
u_\rho(x) = \eta\bigg(\frac{|x|}{\rho}\bigg)\, u(x), \quad u \in C_0,\, 0 < \rho \le \rho_0/2,
\]
and let
\[
C = \set{\pi_\M(u_\rho) : u \in C_0}.
\]
The following lemma was proved in Mosconi et al.\! \cite{MR3530213}.

\begin{lemma}[Mosconi et al.\! {\cite[Proposition 3.2]{MR3530213}}] \label{Lemma 2.10}
The set $C$ is a compact symmetric subset of $\Psi^{\lambda_k + c_{17}\, \rho^{N-sp}}$ for some constant $c_{17} > 0$ and is bounded in $L^\infty(\Omega)$. If $\lambda_k + c_{17}\, \rho^{N-sp} < \lambda_{k+1}$, then $i(C) = k$.
\end{lemma}

\begin{lemma} \label{Lemma 2.11}
For any $w \in \M \setminus C$ with support in $\closure{B_{\rho/2}(0)}$, $\exists\, R > 0$ such that, setting $A = \set{\pi_\M((1 - \tau)\, v + \tau w) : v \in C,\, 0 \le \tau \le 1}$, we have
\[
\sup_{u \in A}\, E_0(Ru) \le 0.
\]
\end{lemma}

\begin{proof}
Let $u = \pi_\M((1 - \tau)\, v + \tau w) \in A$. For $R > 0$,
\[
E_0(Ru) \le \frac{R^p}{p} \int_{\R^{2N}} \frac{|u(x) - u(y)|^p}{|x - y|^{N+sp}}\, dx dy - \frac{R^{p_s^\ast}}{p_s^\ast} \int_\Omega |u|^{p_s^\ast}\, dx = R^p - \frac{R^{p_s^\ast}}{p_s^\ast} \pnorm[p_s^\ast]{u}^{p_s^\ast},
\]
so it suffices to show that $\pnorm[p_s^\ast]{u}$ is bounded away from zero on $A$. By the H\"{o}lder inequality, it is enough to show that $\pnorm[p]{u}$ is bounded away from zero. As in the proof of Lemma \ref{Lemma 2.4}, it suffices to show that $\pnorm[p]{v}$ is bounded away from zero on $C$. Since $C \subset \Psi^{\lambda_k + c_{17}\, \rho^{N-sp}}$ by Lemma \ref{Lemma 2.10}, we have
\[
\pnorm[p]{v}^p = \frac{p}{\Psi(v)} \ge \frac{p}{\lambda_k + c_{17}\, \rho^{N-sp}}. \QED
\]
\end{proof}

Let $\delta \in (0,\lambda_{k+1} - \lambda)$, let $\rho \in (0,\rho_0/2]$ be so small that $\lambda_k + c_{17}\, \rho^{N-sp} < \lambda + \delta$, and let $C_\delta = C$. Then $C_\delta$ is a compact symmetric subset of $\Psi^{\lambda + \delta}$ with $i(C_\delta) = k$ that is bounded in $L^\infty(\Omega)$ by Lemma \ref{Lemma 2.10}. We will show that if $\delta > 0$ is sufficiently small, then $\exists\, w_\delta \in \M \setminus C_\delta$ with support in $\closure{B_{\rho/2}(0)}$ such that, setting $A_\delta = \set{\pi_\M((1 - \tau)\, v + \tau w_\delta) : v \in C_\delta,\, 0 \le \tau \le 1}$, we have
\begin{equation} \label{2.58}
\sup_{u \in A_\delta,\, t \ge 0}\, E_0(tu) < \frac{s}{N}\, S_{N,\,p,\,s}^{N/sp}.
\end{equation}
Then Lemma \ref{Lemma 2.11} will give an $R > 0$ such that \eqref{2.57} holds and complete the proof. We note that \eqref{2.58} is equivalent to
\begin{equation} \label{2.59}
\sup_{v \in C_\delta,\, t, \tau \ge 0}\, E_0(tv + \tau w_\delta) < \frac{s}{N}\, S_{N,\,p,\,s}^{N/sp}.
\end{equation}

In the absence of an explicit formula for a minimizer for $S_{N,\,p,\,s}$ in \eqref{1.21}, we will use certain asymptotic estimates for minimizers obtained in Brasco et al.\! \cite{MR3461371} to choose $w_\delta$. It was shown in \cite{MR3461371} that there exists a nonnegative, radially symmetric, and decreasing minimizer $U(x) = U(r),\, r = |x|$ satisfying
\[
\int_{\R^{2N}} \frac{|U(x) - U(y)|^p}{|x - y|^{N+sp}}\, dx dy = \int_{\R^N} U(x)^{p_s^\ast}\, dx = S_{N,\,p,\,s}^{N/sp}
\]
and
\begin{equation} \label{2.60}
c_{18}\, r^{-(N-sp)/(p-1)} \le U(r) \le c_{19}\, r^{-(N-sp)/(p-1)} \quad \forall r \ge 1
\end{equation}
for some constants $c_{18}, c_{19} > 0$. Then the functions
\[
u_\eps(x) = \eps^{-(N-sp)/p}\, U\bigg(\frac{|x|}{\eps}\bigg), \quad \eps > 0
\]
are also minimizers for $S_{N,\,p,\,s}$ satisfying
\[
\int_{\R^{2N}} \frac{|u_\eps(x) - u_\eps(y)|^p}{|x - y|^{N+sp}}\, dx dy = \int_{\R^N} u_\eps(x)^{p_s^\ast}\, dx = S_{N,\,p,\,s}^{N/sp}
\]
and
\[
\frac{U(\theta r)}{U(r)} \le \frac{c_{19}}{c_{18}}\, \theta^{-(N-sp)/(p-1)} \le \half \quad \forall r \ge 1
\]
if $\theta > 1$ is a sufficiently large constant. Let
\[
u_{\eps,\rho}(x) = \begin{cases}
u_\eps(x) & \text{if } |x| \le \rho\\[10pt]
\dfrac{u_\eps(\rho)\, (u_\eps(x) - u_\eps(\theta \rho))}{u_\eps(\rho) - u_\eps(\theta \rho)} & \text{if } \rho < |x| < \theta \rho\\[10pt]
0 & \text{if } |x| \ge \theta \rho
\end{cases}
\]
and
\[
w_{\eps,\rho}(x) = \frac{u_{\eps,\rho}(x)}{\left(\dint_{\R^N} u_{\eps,\rho}(x)^{p_s^\ast}\, dx\right)^{1/p_s^\ast}}
\]
for $0 < \rho \le \rho_0/2$. Then
\begin{equation} \label{2.61}
\int_{\R^N} w_{\eps,\rho}(x)^{p_s^\ast}\, dx = 1
\end{equation}
and for $\eps \le \rho/2$ we have the estimates
\begin{gather}
\label{2.62} \int_{\R^{2N}} \frac{|w_{\eps,\rho}(x) - w_{\eps,\rho}(y)|^p}{|x - y|^{N+sp}}\, dx dy \le S_{N,\,p,\,s} + c_{20}\, (\eps/\rho)^{(N-sp)/(p-1)},\\[10pt]
\label{2.63} \int_{\R^N} w_{\eps,\rho}^p(x)\, dx \ge \begin{cases}
c_{21}\, \eps^{sp} & \text{if } N > sp^2\\[10pt]
c_{21}\, \eps^{sp} \abs{\log\, (\eps/\rho)} & \text{if } N = sp^2
\end{cases}
\end{gather}
for some constants $c_{20}, c_{21} > 0$ (see Mosconi et al.\! \cite[Lemma 2.7]{MR3530213}). Let
\[
w_\delta = \pi_\M(w_{\eps,\rho/2 \theta}).
\]
Since functions in $C_\delta$ have their supports in $\Omega \setminus B_{3 \rho/4}(0)$, while the support of $w_\delta$ is in $\closure{B_{\rho/2}(0)}$, $w_\delta \in \M \setminus C_\delta$. We will show that \eqref{2.59} holds if $\eps, \rho > 0$ are sufficiently small.

Inequality \eqref{2.59} is equivalent to
\begin{equation} \label{2.64}
\sup_{v \in C_\delta,\, t, \tau \ge 0}\, E_0(tv + \tau w_{\eps,\rho/2 \theta}) < \frac{s}{N}\, S_{N,\,p,\,s}^{N/sp}.
\end{equation}

\begin{lemma} \label{Lemma 2.12}
For $v \in C_\delta$ and $t, \tau \ge 0$,
\[
E_0(tv + \tau w_{\eps,\rho/2 \theta}) \le \left[E_0(tv) + c_{22}\, \rho^{N-sp}\, t^p\right] + \left[E_0(\tau w_{\eps,\rho/2 \theta}) + c_{23}\, (\eps/\rho)^{(N-sp)/(p-1)}\, \tau^p\right]
\]
for some constants $c_{22}, c_{23} > 0$.
\end{lemma}

\begin{proof}
Since $v$ and $w_{\eps,\rho/2 \theta}$ have disjoint supports,
\begin{multline} \label{2.65}
E_0(tv + \tau w_{\eps,\rho/2 \theta}) = \frac{1}{p} \int_{\R^{2N}} \frac{|(tv(x) + \tau w_{\eps,\rho/2 \theta}(x)) - (tv(y) + \tau w_{\eps,\rho/2 \theta}(y))|^p}{|x - y|^{N+sp}}\, dx dy\\[10pt]
- \int_\Omega \left(\frac{\lambda t^p}{p}\, |v|^p + \frac{t^{p_s^\ast}}{p_s^\ast}\, |v|^{p_s^\ast}\right) dx - \int_\Omega \left(\frac{\lambda \tau^p}{p}\, w_{\eps,\rho/2 \theta}^p + \frac{\tau^{p_s^\ast}}{p_s^\ast}\, w_{\eps,\rho/2 \theta}^{p_s^\ast}\right) dx.
\end{multline}
Denote by $I_1$ the first integral on the right-hand side. Since $\supp v \subset B_{3 \rho/4}^c$ and $\supp w_{\eps,\rho/2 \theta} \subset \closure{B_{\rho/2}}$,
\begin{multline} \label{2.66}
I_1 \le t^p \int_{B_{\rho/2}^c \times B_{\rho/2}^c} \frac{|v(x) - v(y)|^p}{|x - y|^{N+sp}}\, dx dy + \tau^p \int_{B_{3 \rho/4} \times B_{3 \rho/4}} \frac{|w_{\eps,\rho/2 \theta}(x) - w_{\eps,\rho/2 \theta}(y)|^p}{|x - y|^{N+sp}}\, dx dy\\[10pt]
+ 2 \int_{B_{3 \rho/4}^c \times B_{\rho/2}} \frac{|tv(x) - \tau w_{\eps,\rho/2 \theta}(y)|^p}{|x - y|^{N+sp}}\, dx dy =: t^p I_2 + \tau^p I_3 + 2 I_4.
\end{multline}

First suppose $p \ge 2$. To estimate $I_4$, we use the elementary inequality
\[
|a + b|^p \le |a|^p + |b|^p + C_p\, (|a|^{p-1} |b| + |a||b|^{p-1}) \quad \forall a, b \in \R
\]
for some constant $C_p > 0$. Since $v(y) = 0$ for $y \in B_{\rho/2}$ and $w_{\eps,\rho/2 \theta}(x) = 0$ for $x \in B_{3 \rho/4}^c$, we get
\begin{multline} \label{2.67}
I_4 \le t^p \int_{B_{3 \rho/4}^c \times B_{\rho/2}} \frac{|v(x) - v(y)|^p}{|x - y|^{N+sp}}\, dx dy + \tau^p \int_{B_{3 \rho/4}^c \times B_{\rho/2}} \frac{|w_{\eps,\rho/2 \theta}(x) - w_{\eps,\rho/2 \theta}(y)|^p}{|x - y|^{N+sp}}\, dx dy\\[10pt]
\hspace{-22.61pt} + C_p\, \Bigg(t^{p-1} \tau \int_{B_{3 \rho/4}^c \times B_{\rho/2}} \frac{|v(x)|^{p-1}\, w_{\eps,\rho/2 \theta}(y)}{|x - y|^{N+sp}}\, dx dy + t \tau^{p-1} \int_{B_{3 \rho/4}^c \times B_{\rho/2}} \frac{|v(x)|\, w_{\eps,\rho/2 \theta}(y)^{p-1}}{|x - y|^{N+sp}}\, dx dy\Bigg)\\[10pt]
=: t^p I_5 + \tau^p I_6 + C_p \left(t^{p-1} \tau J_1 + t \tau^{p-1} J_{p-1}\right),
\end{multline}
where
\[
J_q = \int_{B_{3 \rho/4}^c \times B_{\rho/2}} \frac{|v(x)|^{p-q}\, w_{\eps,\rho/2 \theta}(y)^q}{|x - y|^{N+sp}}\, dx dy, \quad q = 1, p - 1.
\]
Since $C_\delta$ is bounded in $L^\infty(\Omega)$ and
\[
|x - y| \ge |x| - |y| > |x| - \frac{\rho}{2} \ge |x| - \frac{2}{3}\, |x| = \frac{|x|}{3} \quad \forall (x,y) \in B_{3 \rho/4}^c \times B_{\rho/2},
\]
we have
\begin{equation} \label{2.68}
J_q \le c_{24} \int_{B_{3 \rho/4}^c \times B_{\rho/2}} \frac{w_{\eps,\rho/2 \theta}(y)^q}{|x|^{N+sp}}\, dx dy = c_{25}\, \rho^{-sp} \int_{B_{\rho/2}} w_{\eps,\rho/2 \theta}(y)^q\, dy
\end{equation}
for some constants $c_{24}, c_{25} > 0$. By Mosconi et al.\! \cite[Lemma 2.7]{MR3530213}, $|u_{\eps,\rho/2 \theta}|_{p_s^\ast}$ is bounded away from zero and hence
\begin{equation} \label{2.69}
\int_{B_{\rho/2}} w_{\eps,\rho/2 \theta}(y)^q\, dy \le c_{26} \int_{B_{\rho/2}} u_{\eps,\rho/2 \theta}(y)^q\, dy
\end{equation}
for some constant $c_{26} > 0$. Noting that $u_{\eps,\rho/2 \theta} \le u_\eps$, we have
\begin{multline} \label{2.70}
\int_{B_{\rho/2}} u_{\eps,\rho/2 \theta}(y)^q\, dy \le \int_{B_{\rho/2}} u_\eps(y)^q\, dy = \eps^{-(N-sp)\,q/p} \int_{B_{\rho/2}} U\bigg(\frac{|y|}{\eps}\bigg)^q\, dy\\[10pt]
= \eps^{N-(N-sp)\,q/p} \int_{B_{\rho/2 \eps}} U(|y|)^q\, dy.
\end{multline}
When $q < N(p - 1)/(N - sp)$, \eqref{2.60} gives
\begin{equation} \label{2.71}
\int_{B_{\rho/2 \eps}} U(|y|)^q\, dy \le c_{27}\, (\rho/\eps)^{N-(N-sp)\,q/(p-1)}
\end{equation}
for some constant $c_{27} > 0$, in particular, \eqref{2.71} holds for $q = 1$ when $p > 2N/(N + s)$ and for $q = p - 1$. Combining \eqref{2.68}--\eqref{2.71} gives
\[
J_q \le c_{28}\, \rho^{(N-sp)(p-q-1)/(p-1)}\, \eps^{(N-sp)\,q/p\,(p-1)}
\]
for some constant $c_{28} > 0$, so
\begin{multline*}
t^{p-q}\, \tau^q J_q \le c_{28} \left(\rho^{N-sp}\, t^p\right)^{1-q/p} \left((\eps/\rho)^{(N-sp)/(p-1)}\, \tau^p\right)^{q/p} \le c_{29}\, \rho^{N-sp}\, t^p\\[10pt]
+ c_{30}\, (\eps/\rho)^{(N-sp)/(p-1)}\, \tau^p
\end{multline*}
for some constants $c_{29}, c_{30} > 0$ by Young's inequality. Combining this with \eqref{2.65}--\eqref{2.67}, and noting that
\[
I_2 + 2 I_5 = \int_{\R^{2N}} \frac{|v(x) - v(y)|^p}{|x - y|^{N+sp}}\, dx dy, \qquad I_3 + 2 I_6 = \int_{\R^{2N}} \frac{|w_{\eps,\rho/2 \theta}(x) - w_{\eps,\rho/2 \theta}(y)|^p}{|x - y|^{N+sp}}\, dx dy,
\]
we get the desired conclusion in this case.

If $1 < p < 2$, we use the elementary inequality
\[
|a + b|^p \le |a|^p + |b|^p + p\, |a||b|^{p-1} \quad \forall a, b \in \R
\]
to get
\[
I_4 \le t^p I_5 + \tau^p I_6 + p\, t \tau^{p-1} J_{p-1}
\]
and proceed as above.
\end{proof}

By Lemma \ref{Lemma 2.12},
\begin{multline} \label{2.72}
\sup_{v \in C_\delta,\, t, \tau \ge 0}\, E_0(tv + \tau w_{\eps,\rho/2 \theta}) \le \sup_{v \in C_\delta,\, t \ge 0} \left[E_0(tv) + c_{22}\, \rho^{N-sp}\, t^p\right] + \sup_{\tau \ge 0} \big[E_0(\tau w_{\eps,\rho/2 \theta})\\[10pt]
+ c_{23}\, (\eps/\rho)^{(N-sp)/(p-1)}\, \tau^p\big] =: K_1 + K_2.
\end{multline}

\begin{lemma} \label{Lemma 2.13}
We have
\[
K_1 \le \begin{cases}
0 & \text{if } (\lambda_k + c_{17}\, \rho^{N-sp})(1 + c_{22}\, \rho^{N-sp}) \le \lambda < \lambda_{k+1}\\[10pt]
c_{31}\, \rho^{N(N-sp)/sp} & \text{if } \lambda = \lambda_k,
\end{cases}
\]
where $c_{17}$ is as in Lemma \ref{Lemma 2.10} and $c_{31} > 0$ is a constant.
\end{lemma}

\begin{proof}
For $v \in C_\delta$ and $t \ge 0$,
\begin{multline*}
E_0(tv) + c_{22}\, \rho^{N-sp}\, t^p = t^p \left(\frac{1}{p} \int_{\R^{2N}} \frac{|v(x) - v(y)|^p}{|x - y|^{N+sp}}\, dx dy - \frac{\lambda}{p} \int_\Omega |v|^p\, dx + c_{22}\, \rho^{N-sp}\right)\\[10pt]
- \frac{t^{p_s^\ast}}{p_s^\ast} \int_\Omega |v|^{p_s^\ast}\, dx =: K_3\, t^p - K_4\, t^{p_s^\ast},
\end{multline*}
and
\[
K_3 = 1 - \frac{\lambda}{\Psi(v)} + c_{22}\, \rho^{N-sp} \le 1 - \frac{\lambda}{\lambda_k + c_{17}\, \rho^{N-sp}} + c_{22}\, \rho^{N-sp}
\]
since $C_\delta \subset \Psi^{\lambda_k + c_{17}\, \rho^{N-sp}}$ by Lemma \ref{Lemma 2.10}. So $E_0(tv) + c_{22}\, \rho^{N-sp}\, t^p \le 0$ if $(\lambda_k + c_{17}\, \rho^{N-sp})(1 + c_{22}\, \rho^{N-sp}) \le \lambda < \lambda_{k+1}$. If $\lambda = \lambda_k$, then
\[
K_3 \le \frac{c_{17}\, \rho^{N-sp}}{\lambda_k + c_{17}\, \rho^{N-sp}} + c_{22}\, \rho^{N-sp} \le c_{32}\, \rho^{N-sp},
\]
where $c_{32} = c_{17}/\lambda_ k + c_{22} > 0$, and $K_4 \ge c_{33}$ for some constant $c_{33} > 0$ as in the proof of Lemma \ref{Lemma 2.11}, so
\[
E_0(tv) + c_{22}\, \rho^{N-sp}\, t^p \le c_{32}\, \rho^{N-sp}\, t^p - c_{33}\, t^{p_s^\ast}
\]
and maximizing the right-hand side over all $t \ge 0$ gives the desired conclusion.
\end{proof}

\begin{lemma} \label{Lemma 2.13.1}
We have
\[
K_2 \le \begin{cases}
\dfrac{s}{N} \left[S_{N,\,p,\,s} + c_{34}\, (\eps/\rho)^{(N-sp)/(p-1)} - \lambda c_{35}\, \eps^{sp}\right]^{N/sp} & \text{if } N > sp^2\\[10pt]
\dfrac{s}{N} \left[S_{N,\,p,\,s} + c_{34}\, (\eps/\rho)^{sp} - \lambda c_{35}\, \eps^{sp} \abs{\log\, (\eps/\rho)}\right]^{N/sp} & \text{if } N = sp^2
\end{cases}
\]
for some constants $c_{34}, c_{35} > 0$.
\end{lemma}

\begin{proof}
We have
\begin{multline*}
E_0(\tau w_{\eps,\rho/2 \theta}) + c_{23}\, (\eps/\rho)^{(N-sp)/(p-1)}\, \tau^p = \frac{\tau^p}{p}\, \bigg(\int_{\R^{2N}} \frac{|w_{\eps,\rho/2 \theta}(x) - w_{\eps,\rho/2 \theta}(y)|^p}{|x - y|^{N+sp}}\, dx dy\\[10pt]
- \lambda \int_\Omega w_{\eps,\rho/2 \theta}^p\, dx + p\, c_{23}\, (\eps/\rho)^{(N-sp)/(p-1)}\bigg) - \frac{\tau^{p_s^\ast}}{p_s^\ast}
\end{multline*}
by \eqref{2.61}, and maximizing the right-hand side over all $\tau \ge 0$ gives
\[
K_2 = \frac{s}{N}\, \bigg(\int_{\R^{2N}} \frac{|w_{\eps,\rho/2 \theta}(x) - w_{\eps,\rho/2 \theta}(y)|^p}{|x - y|^{N+sp}}\, dx dy - \lambda \int_\Omega w_{\eps,\rho/2 \theta}^p\, dx + p\, c_{23}\, (\eps/\rho)^{(N-sp)/(p-1)}\bigg)^{N/sp},
\]
so the desired conclusion follows from \eqref{2.62} and \eqref{2.63}.
\end{proof}

We can now complete the proof of Theorem \ref{Theorem 1.9}. First suppose $N \ge sp^2$ and $\lambda > \lambda_1$ is not an eigenvalue. Then $\lambda_k < \lambda < \lambda_{k+1}$ for some $k \in \N$. Let $\rho \in (0,\rho_0/2]$ be so small that $(\lambda_k + c_{17}\, \rho^{N-sp})(1 + c_{22}\, \rho^{N-sp}) \le \lambda$. Then \eqref{2.64} follows from \eqref{2.72}, Lemma \ref{Lemma 2.13}, and Lemma \ref{Lemma 2.13.1} for sufficiently small $\eps > 0$.

Now suppose $N\, (N - sp^2) > s^2 p^2$ and $\lambda \ge \lambda_1$. Then $\lambda_k \le \lambda < \lambda_{k+1}$ for some $k \in \N$. We have already considered the case where $N > sp^2$ and $\lambda_k < \lambda < \lambda_{k+1}$, so suppose $\lambda = \lambda_k$. Then
\[
\sup_{v \in C_\delta,\, t, \tau \ge 0}\, E_0(tv + \tau w_{\eps,\rho/2 \theta}) \le \frac{s}{N} \left[S_{N,\,p,\,s} + c_{34}\, (\eps/\rho)^{(N-sp)/(p-1)} - \lambda c_{35}\, \eps^{sp}\right]^{N/sp} + c_{31}\, \rho^{N(N-sp)/sp}
\]
by \eqref{2.72}, Lemma \ref{Lemma 2.13}, and Lemma \ref{Lemma 2.13.1}. Set $\rho = \eps^\alpha$, where $\alpha > 0$ is to be chosen. Then the right-hand side is less than or equal to
\[
\frac{s}{N}\, S_{N,\,p,\,s}^{N/sp} \left[1 + c_{36}\, \eps^{(1 - \alpha)(N-sp)/(p-1)} - c_{37}\, \eps^{sp}\right]^{N/sp} + c_{31}\, \eps^{\alpha N(N-sp)/sp}
\]
for some constants $c_{36}, c_{37} > 0$, so \eqref{2.64} will follow for sufficiently small $\eps > 0$ if $\alpha$ can be found so that
\[
(1 - \alpha)(N - sp)/(p - 1) > sp
\]
and
\[
\alpha N(N-sp)/sp > sp.
\]
This is possible if and only if
\[
s^2 p^2/N(N - sp) < (N - sp^2)/(N - sp),
\]
i.e.,
\[
N\, (N - sp^2) > s^2 p^2. \hquad \qedsymbol
\]

\subsection{Proof of Theorem \ref{Theorem 1.10}}

We prove Theorem \ref{Theorem 1.10} by applying Theorem \ref{Theorem 1.5} with $W = W^{1,\,p}_0(\Omega)$ and the operators $\Ap, \Bp, f, g \in C(W^{1,\,p}_0(\Omega),W^{-1,\,p'}(\Omega))$ and $h \in W^{-1,\,p'}(\Omega)$ given by
\begin{multline*}
\dualp{\Ap[u]}{v} = \int_\Omega |\nabla u|^{p-2}\, \nabla u \cdot \nabla v\, dx, \quad \dualp{\Bp[u]}{v} = \int_\Omega |u|^{p-2}\, uv\, dx,\\[10pt]
\dualp{f(u)}{v} = \int_\Omega |u|^{p^\ast - 2}\, uv\, dx, \quad \dualp{g(u)}{v} = - \int_\Omega |\nabla u|^{q-2}\, \nabla u \cdot \nabla v\, dx, \quad u, v \in W^{1,\,p}_0(\Omega)
\end{multline*}
and
\[
\dualp{h}{v} = \int_\Omega h(x)\, v\, dx, \quad v \in W^{1,\,p}_0(\Omega).
\]

\begin{lemma} \label{Lemma 3.17}
There exists $\kappa > 0$ such that the functional $E$ in \eqref{1.23} satisfies the {\em \PS{c}} condition for all
\begin{equation} \label{3.69}
c < \frac{1}{N}\, S_{N,\,p}^{N/p} - \kappa \pnorm[{p^\ast}']{h}^{{p^\ast}'}.
\end{equation}
\end{lemma}

\begin{proof}
Let $c \in \R$ and let $\seq{u_j}$ be a sequence in $W^{1,\,p}_0(\Omega)$ such that
\begin{equation} \label{3.70}
E(u_j) = \int_\Omega \left(\frac{1}{p}\, |\nabla u_j|^p + \frac{\mu}{q}\, |\nabla u_j|^q - \frac{\lambda}{p}\, |u_j|^p - \frac{1}{p^\ast}\, |u_j|^{p^\ast} - h(x)\, u_j\right) dx = c + \o(1)
\end{equation}
and
\begin{multline} \label{3.71}
\dualp{E'(u_j)}{v} = \int_\Omega \big(|\nabla u_j|^{p-2}\, \nabla u_j \cdot \nabla v + \mu\, |\nabla u_j|^{q-2}\, \nabla u_j \cdot \nabla v - \lambda\, |u_j|^{p-2}\, u_j\, v\\[10pt]
- |u_j|^{p^\ast - 2}\, u_j\, v - h(x)\, v\big)\, dx = \o(\norm{v}) \quad \forall v \in W^{1,\,p}_0(\Omega).
\end{multline}
Taking $v = u_j$ in \eqref{3.71} gives
\begin{equation} \label{3.72}
\int_\Omega \left(|\nabla u_j|^p + \mu\, |\nabla u_j|^q - \lambda\, |u_j|^p - |u_j|^{p^\ast} - h(x)\, u_j\right) dx = \o(\norm{u_j}).
\end{equation}
Let $r \in (p,p^\ast)$. Dividing \eqref{3.72} by $r$ and subtracting from \eqref{3.70} gives
\begin{multline*}
\int_\Omega \bigg[\left(\frac{1}{p} - \frac{1}{r}\right) |\nabla u_j|^p + \mu \left(\frac{1}{q} - \frac{1}{r}\right) |\nabla u_j|^q - \lambda \left(\frac{1}{p} - \frac{1}{r}\right) |u_j|^p + \left(\frac{1}{r} - \frac{1}{p^\ast}\right) |u_j|^{p^\ast}\\[10pt]
- \left(1 - \frac{1}{r}\right) h(x)\, u_j\bigg] dx = c + \o(1) + \o(\norm{u_j}),
\end{multline*}
and it follows from this that $\seq{u_j}$ is bounded. So a renamed subsequence converges to some $u$ weakly in $W^{1,\,p}_0(\Omega)$, strongly in $L^t(\Omega)$ for all $t \in [1,p^\ast)$, and a.e.\! in $\Omega$. Setting $\widetilde{u}_j = u_j - u$, we will show that $\widetilde{u}_j \to 0$ in $W^{1,\,p}_0(\Omega)$.

Equation \eqref{3.72} gives
\begin{equation} \label{3.73}
\norm{u_j}^p + \mu \pnorm[q]{\nabla u_j}^q = \pnorm[p^\ast]{u_j}^{p^\ast} + \int_\Omega \left(\lambda\, |u|^p + h(x)\, u\right) dx + \o(1).
\end{equation}
Taking $v = u$ in \eqref{3.71} and passing to the limit gives
\begin{equation} \label{3.74}
\norm{u}^p + \mu \pnorm[q]{\nabla u}^q = \pnorm[p^\ast]{u}^{p^\ast} + \int_\Omega \left(\lambda\, |u|^p + h(x)\, u\right) dx.
\end{equation}
Since
\begin{equation} \label{3.75}
\norm{\widetilde{u}_j}^p = \norm{u_j}^p - \norm{u}^p + \o(1)
\end{equation}
and
\[
\pnorm[p^\ast]{\widetilde{u}_j}^{p^\ast} = \pnorm[p^\ast]{u_j}^{p^\ast} - \pnorm[p^\ast]{u}^{p^\ast} + \o(1)
\]
by the Br{\'e}zis-Lieb lemma \cite[Theorem 1]{MR699419}, and
\begin{equation} \label{3.77}
\liminf \pnorm[q]{\nabla u_j} \ge \pnorm[q]{\nabla u},
\end{equation}
\eqref{3.73} and \eqref{3.74} imply
\[
\norm{\widetilde{u}_j}^p \le \pnorm[p^\ast]{\widetilde{u}_j}^{p^\ast} + \o(1) \le \frac{\norm{\widetilde{u}_j}^{p^\ast}}{S_{N,\,p}^{p^\ast/p}} + \o(1),
\]
so
\begin{equation} \label{3.76}
\norm{\widetilde{u}_j}^p \left(S_{N,\,p}^{N/(N-p)} - \norm{\widetilde{u}_j}^{p^2/(N-p)}\right) \le \o(1).
\end{equation}
On the other hand, \eqref{3.70} gives
\[
c = \frac{1}{p} \norm{u_j}^p + \frac{\mu}{q} \pnorm[q]{\nabla u_j}^q - \frac{1}{p^\ast} \pnorm[p^\ast]{u_j}^{p^\ast} - \int_\Omega \left(\frac{\lambda}{p}\, |u|^p + h(x)\, u\right) dx + \o(1),
\]
and a straightforward calculation combining this with \eqref{3.73}--\eqref{3.75} gives
\[
c = \frac{1}{N} \norm{\widetilde{u}_j}^p + \mu \left[\left(\frac{1}{q} - \frac{1}{p^\ast}\right) \pnorm[q]{\nabla u_j}^q - \frac{1}{N} \pnorm[q]{\nabla u}^q\right] + \int_\Omega \left[\frac{1}{N}\, |u|^{p^\ast} - \left(1 - \frac{1}{p}\right) h(x)\, u\right] dx + \o(1).
\]
The second term on the right-hand side is greater than or equal to $\o(1)$ by \eqref{3.77} and the integral is greater than or equal to
\[
\frac{1}{N} \pnorm[p^\ast]{u}^{p^\ast} - \left(1 - \frac{1}{p}\right) \pnorm[{p^\ast}']{h} \pnorm[p^\ast]{u} \ge - \kappa \pnorm[{p^\ast}']{h}^{{p^\ast}'}
\]
for some $\kappa > 0$ by the H\"{o}lder and Young's inequalities, so
\[
\norm{\widetilde{u}_j}^p \le N \left(c + \kappa \pnorm[{p^\ast}']{h}^{{p^\ast}'}\right) + \o(1).
\]
Combining this with \eqref{3.76} shows that $\widetilde{u}_j \to 0$ when \eqref{3.69} holds.
\end{proof}

We apply Theorem \ref{Theorem 1.5} with
\[
c_{\mu,\,h}^\ast = \frac{1}{N}\, S_{N,\,p}^{N/p} - \kappa \pnorm[{p^\ast}']{h}^{{p^\ast}'},
\]
where $\kappa > 0$ is as in Lemma \ref{Lemma 3.17}, noting that
\[
\lim_{\mu,\, \pnorm[{p^\ast}']{h} \to 0}\, c_{\mu,\,h}^\ast = \frac{1}{N}\, S_{N,\,p}^{N/p}.
\]
The set $\M$ and the functions $\Psi$, $\pi_\M$, and $E_0$ are the same as in the proof of Theorem \ref{Theorem 1.6}. Let $\lambda_k \le \lambda < \lambda_{k+1}$. Exactly as in that proof, there exist $R > 0$ and, for all sufficiently small $\delta > 0$, a compact symmetric subset $C_\delta$ of $\Psi^{\lambda + \delta}$ with $i(C_\delta) = k$ and $w_\delta \in \M \setminus C_\delta$ such that, setting $A_\delta = \set{\pi_\M((1 - \tau)\, v + \tau w_\delta) : v \in C_\delta,\, 0 \le \tau \le 1}$, we have
\[
\sup_{u \in A_\delta}\, E_0(Ru) \le 0, \qquad \sup_{u \in A_\delta,\, 0 \le t \le R}\, E_0(tu) < \frac{1}{N}\, S_{N,\,p}^{N/p}. \hquad \qedsymbol
\]

\def\cdprime{$''$}

\end{document}